\documentclass{amsart}

\usepackage[utf8]{inputenc}
\usepackage[T1]{fontenc}
\usepackage{a4wide}
\usepackage{graphicx}
\usepackage{color}
\usepackage[shortlabels]{enumitem}

\usepackage{fixme}
\fxsetup{theme=color}

\usepackage{hyperref}

\usepackage{amsfonts}         % Pour l'usage de \mathbb, \mathfrak, \mathscr
\usepackage{amssymb}          % et de divers symboles supplémentaires
\usepackage[mathscr]{eucal}   % Permet l'usage des polices CM et Euler 

\newcommand{\R}{\mathbb{R}}

\newcommand{\om}{\omega}
\newcommand{\Om}{\Omega}

\newcommand{\trans}{\intercal}

\DeclareMathOperator{\gr}{gr}
\DeclareMathOperator{\PD}{PD}
\DeclareMathOperator{\id}{id}
\DeclareMathOperator{\supp}{supp}
\DeclareMathOperator{\dist}{dist}

\DeclareMathOperator{\SO}{SO}

\DeclareMathOperator{\Diff}{Diff}

\newcommand{\CP}{\mathbb{C}P}

\newtheorem{theorem}{Theorem}[section]
\newtheorem{prop}[theorem]{Proposition}
\newtheorem{lemma}[theorem]{Lemma}
\newtheorem{cor}[theorem]{Corollary}

\newtheorem{remark}[theorem]{Remark}

\numberwithin{equation}{section}
\numberwithin{figure}{section}

\newcommand{\eoesymbol}{$\between$}

\DeclareRobustCommand{\eoe}{%
  \ifmmode \mathqed
  \else
    \leavevmode\unskip\penalty9999 \hbox{}\nobreak\hfill
    \quad\hbox{\eoesymbol}%
  \fi
}

\begin{document}

%%%%%%%%%%%%%%%%%%%%%%%%%%%%%%%%%%%%%%%%%%%%%%%%%%%%%%%%%%%%%%%%%%%%%%%%%%%%%%%%
% Front matter
%%%%%%%%%%%%%%%%%%%%%%%%%%%%%%%%%%%%%%%%%%%%%%%%%%%%%%%%%%%%%%%%%%%%%%%%%%%%%%%%

\title[$J$-tame inflation]{J-tamed inflation via tame to compatible deformations}

\author[P. Chakravarthy]{Pranav Chakravarthy}
\address{PC: Department of Mathematics\\Universit\'e Libre de Bruxelles \\ Brussels, Belgium}
\email{pranav.vijay.chakravarthy@ulb.be}

\author[J. Payette]{Jordan Payette}
\address{JP: Department of Mathematics and Statistics, McGill University \\ Montréal, Québec, Canada}
\email{jordan.payette@mail.mcgill.ca}

\author[M. Pinsonnault]{Martin Pinsonnault}
\address{MP: Department of Mathematics \\ University of Western Ontario \\ London, Ontario, Canada}
\email{mpinson@uwo.ca}

\begin{abstract}
We give a complete and self-contained exposition of the $J$-tame inflation lemma: Given any tame almost complex structure $J$ on a symplectic $4$-manifold $(M,\omega)$, and given any compact, embedded, $J$-holomorphic submanifold $Z$, it is always possible to construct a deformation of symplectic forms $\omega_t$ in classes $[\omega_t]=[\omega]+t\mathrm{PD}{Z}$, for $0\leq t$ less than an upper bound $0<T$ that only depends on the self-intersection $Z\cdot Z$. The original proofs of this fact make the unwarranted assumption that one can find a family of normal planes along $Z$ that is both $J$ invariant and $\omega$-orthogonal to $TZ$ --- which amounts, in effect, to assuming the compatibility of $J$ and $\omega$ along $Z$. We explain how the original constructions can be adapted to avoid this assumption when $Z$ has nonpositive self-intersection, and we discuss the difficulties with this line of argument in general to establish the full inflation when $Z$ has positive self-intersection. We overcome this problem by proving a `preparation lemma', which states that prior to inflation, one can isotope $\omega$ within its cohomology class to a new form that still tames $J$ and which is compatible with $J$ along the submanifold $Z$. This preparation lemma can be regarded as an infinitesimal version of the "tamed-to-compatible" conjecture of S. K. Donaldson along an almost-complex submanifold $Z$.
\end{abstract}

\thanks{PC is thankful to Universit\'e Libre de Bruxelles where this project was completed, supported by the FWO and the FNRS via EOS project 40007524. PC is also grateful to the Hebrew University of Jerusalem for a postdoctoral fellowship supported by the grants of ISF (grant no. 2445/20) and the BSF (grant no. 2020310) where part of the work was undertaken. JP was supported by a FRQNT postdoctoral scholarship. MP is supported by NSERC Discovery Grant RGPIN-2020-06428. MP would like to thank the CRM where part of this work was completed.}
\maketitle

\tableofcontents

%%%%%%%%%%%%%%%%%%%%%%%%%%%%%%%%%%%%%%%%%%%%%%%%%%%%%%%%%%%%%%%%%%%%%%%%%%%%%%%%
\section{Introduction}
%%%%%%%%%%%%%%%%%%%%%%%%%%%%%%%%%%%%%%%%%%%%%%%%%%%%%%%%%%%%%%%%%%%%%%%%%%%%%%%%

%%%%%%%%%%%%%%%%%%%%%%%%%%%%%%%%%%%%%%%%%%%%%%%%%%%%%%%%%%%%%%%%%%%%%%%%%%%%%%%%
\subsection{The inflation lemma}
%%%%%%%%%%%%%%%%%%%%%%%%%%%%%%%%%%%%%%%%%%%%%%%%%%%%%%%%%%%%%%%%%%%%%%%%%%%%%%%%
Let $(M,\om)$ be a symplectic $4$-manifold and let $Z\subset M$ be a compact, embedded, symplectic submanifold without boundary. The inflation lemma states that we can always deform the symplectic form $\om$ in the direction of $\PD[Z]$ within the symplectic cone of $M$, and that the size of the deformation only depends on the self-intersection of $Z$. This was first stated in~\cite{La-Isotopy} in the case $Z\cdot Z\geq 0$. Since then, different versions of the inflation lemma appeared in the litterature, differing mainly in the compatibility condition one imposes between the deformed symplectic forms and an auxiliary almost complex structure $J$ for which $Z$ is $J$-holomorphic. In this paper, we are mainly concerned with the so called $J$-tame, or "tame-to-tame", inflation lemma, namely,
\begin{lemma}[Tame-to-tame inflation~\cite{McDuff-J-inflation, Buse-Neg-inflation}]\label{Lemma:Jtameinflation}
Let $(M,\om)$ be a symplectic $4$-manifold, $J$ a tame almost complex structure, and let $Z\subset M$ be a compact, embedded, $J$-holomorphic curve without boundary. There exist symplectic forms taming $J$ in class $[\om]+t\PD[Z]$, for $t\in[0,T)$, where
\[T=
\begin{cases}
\infty& \text{~if~} Z\cdot Z\geq 0\\
-\frac{\om(Z)}{Z\cdot Z}& \text{~if~} Z\cdot Z < 0.
\end{cases}
\]
\end{lemma}
The proofs of the tame-to-tame inflation lemma given in~\cite{McDuff-J-inflation} and~\cite{Buse-Neg-inflation} both assume that the symplectic normal bundle $\nu_{Z}$ along the submanifold $Z$ is $J$-invariant. But this is the case if, and only if, $J$ is compatible with $\om$ on $T_{Z}M$. In this situation, $\nu_{Z}$ coincides with the Riemannian normal bundle defined by the associated metric $g_{J}(x,y):=\om(x,Jy)$. However, even if we start from an $\om$ compatible almost complex structure $J$, the inflated form in class $[\om]+t \PD[Z]$ contructed in~\cite{McDuff-J-inflation, Buse-Neg-inflation} only tames $J$. We call this weaker version of inflation "compatible-to-tame". Whilst this statement suffices for applications such as McDuff's "deformation to isotopy lemma"~\cite{McDuff-DeformationToIsotopy}, the compatible-to-tame inflation is insufficient for showing the stability of homotopy type of symplectomorphism groups as done in \cite{McDuff-J-inflation, P08i} and in many other papers that use similar inflation arguments. Nevertheless, as explained in~\cite{ALLP}, the proofs in the above papers can be salvaged by relying on a weaker version of the tame inflation lemma based on Li-Zhang’s comparison of J-symplectic cones~\cite{LZ09} and which holds for manifolds with $b_2^+=1$.
\begin{lemma}[Weak $b_2^+=1$ $J$-compatible inflation~\cite{ALLP}]\label{lemma:Weak inflation}
Let $M$ be a symplectic $4$-manifold with $b_2^+=1$. Given a compatible pair $(J,\om)$ and a $J$-holomorphic embedded curve $Z$, there exists a symplectic form $\om'$ compatible with $J$ such that $[\om']= [\om]+ t \PD(Z)$, $t\in [0,T)$ where $T= \infty$ if $Z\cdot Z\geq 0$ and $T= \frac{\om(Z)}{-Z\cdot Z}$ if $Z\cdot Z<0$.
\end{lemma}
Although this version currently suffices for most current applications, it relies on global results about the symplectic cone, a situation that considerably limits the applicability of inflation arguments. Therefore, it is highly desirable to prove the tame-to-tame inflation lemma using only local arguments that hold for all $4$-manifolds.

As explained below, it is easy to fix the original proofs of the tame-to-tame inflation lemma in the cases $Z\cdot Z=0$ and $Z\cdot Z<0$. The proofs consist in constructing a suitable representative of the Thom class of $Z$. These direct arguments have the advantage of being adaptable to more general situations such as normal crossing divisors with $\om$-orthogonal crossings or higher dimensional hypersurfaces. This is why we provide a complete exposition. Unfortunately, in the case of curves of strictly positive self-intersection $Z\cdot Z=m>0$, the original approach imposes an upper bound on the inflation parameter and only produces symplectic forms in classes
\[
[\om]+t \PD[Z] \text{~for~} t<T\sim C(J)\cdot\frac{1}{m}
\]
where $C(J)$ is a constant that approaches $0$ as $J$ moves further away from being compatible with $\om$ along $Z$. This is why, in the case $Z\cdot Z>0$, we take another route and show that the tame-to-tame inflation process can be performed in two steps. First, we isotope the symplectic form $\om$ near the embedded $J$-holomorphic curve $Z$ so that $J$ is compatible with the new form along $Z$. Then, we apply the compatible-to-tame inflation process. The main result of this paper is thus the following "preparation lemma" that allows us to perform the first step.
\begin{lemma}[Preparation lemma]\label{PreparationLemma}
Let $J$ be an almost complex structure tamed by a symplectic form $\om$ on a $4$-manifold $M$. Given an embedded $J$-holomorphic curve $Z$, we can isotope $\om$ (within its cohomology class) to a new form $\om'$ taming $J$ and such that $J$ and $\om'$ are compatible along~$Z$.
\end{lemma}
This approach, which only relies on local considerations, establishes the tame-to-tame inflation process in full generality and has the advantage of working along embedded $J$-holomorphic curves of arbitrary self-intersections. As an immediate corollary of the tame-to-tame inflation, we obtain an equivariant version of inflation, namely,

\begin{cor}[Equivariant $J$-inflation lemma]
Suppose $(M,\om)$ is equipped with a symplectic group action of a compact group $G$. Let $Z$ be a $G$-invariant embedded symplectic submanifold of $(M,\om)$ such that $Z$ is holomorphic with respect to an invariant $\om$-tame almost complex structure $J$. Then there exists a family of invariant symplectic forms taming $J$ in class $[\omega]+t\PD[Z]$, for $\lambda \in[0, T)$, where
\[
T= \begin{cases}\infty & \text { if } Z \cdot Z \geq 0 \\
-\frac{\omega(Z)}{Z \cdot Z} & \text { if } Z \cdot Z<0\end{cases}.
\]
\end{cor}
\begin{proof}
By Lemma~\ref{Lemma:Jtameinflation}, there is a path $\om_t$ of non-invariant symplectic forms in  cohomology classes $[\omega]+t\PD[Z]$, for $t<T$. Averaging $\om_t$ under the group action defines a new path $\tilde{\omega}_t:= \int_G \om_t dg$. Since the $G$ action preserves the classes   $[\om]$ and $[Z]$, as well as the orientation class $[\om]^2$, it also preserves $\PD[Z]$. Consequently,  $[\tilde{\om}_t] = [\om_t]$. Furthermore, as $J$ is invariant under the $G$ action, $\tilde{\om}_\lambda$ is a path of invariant symplectic forms taming $J$.
\end{proof}

\begin{remark}
Recall that Donaldson's "tame-to-compatible" conjecture~\cite{Do2006} states that on any $4$-manifold, if an almost complex structure $J$ is tamed by a symplectic form $\om$, then there exists a cohomologous symplectic form $\om'$ that is compatible with $J$. Clearly, the Preparation lemma~\ref{PreparationLemma} follows directly from Donaldson's conjecture, and it seems relevant to list $4$-manifolds for which it is known to hold. The first breakthrough came from the work of Taubes~\cite{Taubes} who showed that the conjecture holds for \emph{generic} $J$ on surfaces with $b_{2}^{+}=1$. Later, Li and Zhang~\cite{TJL-Zhang} proved that it holds for $\CP^{2}$, for rational ruled surfaces, and for rational surfaces with $b_2-1$ disjoint exceptional $J$-curves. To our knowledge, the most recent advance on this problem is the work of Tan and al.~\cite{TQWZZ} which proves Donaldson's conjecture for $4$-manifolds satisfying a specific Hodge-theoretic condition, including all $4$-manifolds with $b^{+}_{2}=1$. Combining this later result with the work of McDuff~\cite{McDuff-Rational+Ruled} showing that the only symplectic $4$-manifolds that contain embedded spheres of non-negative self-intersections are blow-ups of rational or ruled manifolds, we conclude that the $J$-tame inflation holds for any embedded sphere $Z$ with $Z\cdot Z\geq 0$. Since the tame-to-tame inflation lemma holds whenever $Z\cdot Z\leq 0$, it follows that the only cases remaining are surfaces $Z$ of genus $g(Z)\geq 1$ and of self-intersection $Z\cdot Z\geq 1$. For all these cases, the preparation lemma is still essential.\eoe
\end{remark}

%%%%%%%%%%%%%%%%%%%%%%%%%%%%%%%%%%%%%%%%%%%%%%%%%%%%%%%%%%%%%%%%%%%%%%%%%%%%%%%%
\section{Preliminary considerations}\label{section:preliminary considerations}
%%%%%%%%%%%%%%%%%%%%%%%%%%%%%%%%%%%%%%%%%%%%%%%%%%%%%%%%%%%%%%%%%%%%%%%%%%%%%%%%
We begin with some elementary observations that will be useful in restating the tameness and compatibility conditions of $J$ along $Z$ in terms of suitable norms. This will be used in the proof of the tame-to-tame inflation lemma in the case $Z\cdot Z\leq 0$ given in Section~\ref{section:Thom forms}, as well as in the proof of the preparation lemma that occupies the rest of the document.

%%%%%%%%%%%%%%%%%%%%%%%%%%%%%%%%%%%%%%%%%%%%%%%%%%%%%%%%%%%%%%%%%%%%%%%%%%%%%%%%
\subsection{Spaces of tame and compatible pairs \texorpdfstring{$(\omega, J)$}{w}}\label{subsec:tame_comp}
%%%%%%%%%%%%%%%%%%%%%%%%%%%%%%%%%%%%%%%%%%%%%%%%%%%%%%%%%%%%%%%%%%%%%%%%%%%%%%%%
Let $V$ be a $2n$-dimensional real vector space. Let $J$ be a complex structure on $V$, \emph{i.e.}, $J^2 = -1$. We say that a 2-form $\omega$ on $V$:
\begin{enumerate}
\item \emph{tames} or \emph{is adapted to} $J$ if the 2-tensor $\omega \circ (1 \otimes J)$ is positive-definite, \emph{i.e.}, if $v \in \R^{2n} \setminus \{ 0 \}$ implies $\omega( v, Jv) > 0$,

\item  that it is \emph{$J$-invariant} if $\omega \circ  (J \otimes J) = \omega$, \emph{i.e} if $\omega(J v, J w) = \omega( v,w)$ for all $v,w \in \R^{2n}$,

\item  and that it is \emph{compatible with} $J$ if it is both adapted to $J$ and $J$-invariant.
\end{enumerate}

\noindent We note that a 2-form $\omega$ that is adapted to $J$ is automatically nondegenerate (hence $\omega$ is a symplectic 2-form) and that the tameness condition is an open one. \vspace{0pt}

The sets
\[
\Om_{\tau}(J) := \{ \om~|~\omega \text{~tames~} J\} \text{~and~} \Om_c(J) := \{\om~|~ \omega \text{~is compatible with~} J\}
\]
are easily seen to be convex. We note that $\Om_c(J)$ is the fixed-point set of the involution
\[ \iota : \Om_{\tau}(J) \to \Om_{\tau}(J) : \omega \mapsto \iota(\omega) := \omega \circ (J \otimes J) \; . \]
This map is well-defined: $\iota(\om)$ is clearly a 2-form and the calculation
\[ \iota(\omega)(v, Jv) = \omega(Jv, J^2 v) = - \omega(Jv, v) = \omega(v, Jv) \]
shows that $\iota(\om)$ tames $J$ if (and only if) $\om$ does.

The map
\[ \pi : \Om_{\tau}(J) \to \Om_c(J) : \om \mapsto \frac{1}{2}(\om + \iota(\om)) \]
is a "conical fibration" in the sense that its fibers are (double) cones. Indeed, for $\om \in \Om_{\tau}(J)$ and $t \in [0,1]$, if we set $\om_t := (1-t)\om + t \iota(\om)$, then we have $\om_t \in \Om_{\tau}(J)$, $\iota(\om_t) = \om_{1-t}$ and $\pi(\om_t) =\om_{1/2}$.

As a result, we see that $\Om_{\tau}(J)$ deformation retracts onto $\Om_c(J)$. We also observe that for $\om \in \Om_{\tau}(J)$, we have $\om \in \Om_c(J)$ if and only if $\om_t \equiv \om$ for all $t \in [0,1]$.

Finally, we note that $\Om_{\tau}(J) = \Om_c(J)$ when $V$ is 2-dimensional, \emph{i.e.}, any tame pair $(\om, J)$ is in fact compatible, as is easily seen by considering a basis $\langle v, Jv \rangle$ of $V$ such that $\om(v, Jv) = 1$. \vspace{0pt}

%%%%%%%%%%%%%%%%%%%%%%%%%%%%%%%%%%%%%%%%%%%%%%%%%%%%%%%%%%%%%%%%%%%%%%%%%%%%%%%%
\subsection{Symplectic splitting}\label{section:symplectic splitting}
%%%%%%%%%%%%%%%%%%%%%%%%%%%%%%%%%%%%%%%%%%%%%%%%%%%%%%%%%%%%%%%%%%%%%%%%%%%%%%%%
We shall be primarily concerned with a four dimensional vector space $V$ equipped with a tame linear pair $(\om, J)$. We also suppose that a $J$-invariant 2-dimensional subspace $V_1 \subset V$ is given. (Typically, $V$ will be the tangent space $T_p M$ at $p\in Z$, and $V_1$ will be the subspace $T_pZ$.) We denote $j_1 : V_1 \subset V$ the canonical injection, $J_1 := J|_{V_1}$ the induced complex structure and $\omega_1 := j_1^* \omega$ the induced symplectic form. It is clear that $(\omega_1, J_1)$ is tame (hence compatible) on $V_1$. We also consider the symplectic orthogonal subspace $V_2 := V_1^{\omega}$ of $V_1$, and write $j_2 : V_2 \to V$ for the canonical injection and $\omega_2 = j_2^* \omega$ for the induced symplectic form. $V$ splits as a direct sum $V_1 \oplus V_2$, and we denote $\pi_k : V \to V_k$ the corresponding canonical projections. It follows that $\omega = \pi_1^*\omega_1 + \pi_2^* \omega_2$.

%%%%%%%%%%%%%%%%%%%%%%%%%%%%%%%%%%%%%%%%%%%%%%%%%%%%%%%%%%%%%%%%%%%%%%%%%%%%%%%%
\subsection{Working in a symplectic basis}
%%%%%%%%%%%%%%%%%%%%%%%%%%%%%%%%%%%%%%%%%%%%%%%%%%%%%%%%%%%%%%%%%%%%%%%%%%%%%%%%
It is convenient to work in a symplectic basis of $(V, \omega)$ that is compatible with the splitting $V = V_1 \oplus V_2$.  With respect to such a basis, $\omega$ and $J$ are represented by the block-matrices
\begin{align}\label{eq:matrixform}
\om = \begin{pmatrix} \omega_1 & 0 \\ 0 & \omega_2  \end{pmatrix} \text{~and~} J = \begin{pmatrix} J_1 &  B \\ 0 & J_2  \end{pmatrix} \, , \; \text{~where~} \omega_1 = \omega_2 = J_0^T = \begin{pmatrix} 0 & 1 \\ -1 & 0  \end{pmatrix} \, . 
\end{align}

\noindent The condition $J^2 = -1$ amounts to the system of equations
\begin{align}\label{eq:system} 
J_1^2 = J_2^2 = -1 \; \text{~and~} \; J_1 B + BJ_2 = 0 \, .
\end{align}
Hence the matrix $J_2$ determines a complex structure (also denoted) $J_2$ on  $V_2$. Like the complex structure $J_1$ on $V_1$, the complex structure $J_2$ on $V_2$ admits a basis-free expression, namely $J_k = \pi_k \circ J \circ j_k$ for $k=1,2$. (This is most easily proved by representing the $\pi_k$'s and the $j_k$'s as block-matrices.)

%%%%%%%%%%%%%%%%%%%%%%%%%%%%%%%%%%%%%%%%%%%%%%%%%%%%%%%%%%%%%%%%%%%%%%%%%%%%%%%%
\subsection{Canonical scalar products}\label{sec:tion:canonical_metric}
%%%%%%%%%%%%%%%%%%%%%%%%%%%%%%%%%%%%%%%%%%%%%%%%%%%%%%%%%%%%%%%%%%%%%%%%%%%%%%%%
In this particular setting, the tame pair $(\omega, J)$ determines \emph{two} canonical (\emph{i.e.}, basis-independent) scalar products on $V$. The first one, denoted $g_J$, is the one usually associated to such a pair, namely, the symmetrization
\[g_J(v,w) = \frac{1}{2}\Big(\om(v,Jw)+\om(w,Jv)\Big).\]
The second one, simply denoted $g$, is defined in terms of the symplectic splitting. Since the pair $(\omega, J)$ is tame, it follows that the pairs $(\omega_1, J_1)$ and $(\omega_2, J_2)$ are tame, and thus compatible, on $V_1$ and $V_2$ respectively. By compatibility, the 2-tensors $g_1 = \om_1 \circ (1 \otimes J_1)$ and  $g_2 = \om_2 \circ (1 \otimes J_2)$ are scalar products on $V_1$ and $V_2$. We define the scalar product $g = g_1 \oplus g_2$ on $V = V_1 \oplus V_2$, which is thus represented in the above basis by the matrix
\begin{align}\label{eq:matrixform2} 
g = \begin{pmatrix} J_0^T J_1 & 0 \\ 0 & J_0^T J_2  \end{pmatrix} \, . 
\end{align}
We stress that from its very definition, the scalar product $g$ on $V$ is invariantly defined, in the sense that it does not depend on the precise symplectic basis we chose in the previous paragraph. Moreover, \ref{eq:matrixform2} shows that $g$ depends only on the splitting of $V$, on $\omega$ and on the complex structures $J_1$ and $J_2$ induced by $J$. Remarkably, and that is a crucial point in the whole discussion, $g$ does not depend on the "skew" part of $J$ given by the operator $B = \pi_1 \circ J \circ j_2 : V_2 \to V_1$. \vspace{12pt}

%%%%%%%%%%%%%%%%%%%%%%%%%%%%%%%%%%%%%%%%%%%%%%%%%%%%%%%%%%%%%%%%%%%%%%%%%%%%%%%%
\subsection{Working in a unitary basis}\label{subsubsec:unitary-basis}
%%%%%%%%%%%%%%%%%%%%%%%%%%%%%%%%%%%%%%%%%%%%%%%%%%%%%%%%%%%%%%%%%%%%%%%%%%%%%%%%
It is convenient to work with symplectic bases of $V_1$ and $V_2$ such that $J_1 = J_2 = J_0$ as matrices. This can be done since the pairs $(\omega_1, J_1)$ and $(\omega_2, J_2)$ are compatible. The metric $g$ is then represented by the identity matrix, which shows that the choices of such symplectic bases of $V_1$ and $V_2$ are uniquely determined up to rotations in each $V_k$ separately.

%%%%%%%%%%%%%%%%%%%%%%%%%%%%%%%%%%%%%%%%%%%%%%%%%%%%%%%%%%%%%%%%%%%%%%%%%%%%%%%%
\subsection{Norm of the skew part}\label{subsubsec:norm}
%%%%%%%%%%%%%%%%%%%%%%%%%%%%%%%%%%%%%%%%%%%%%%%%%%%%%%%%%%%%%%%%%%%%%%%%%%%%%%%%
Interpreting the skew part $B$ of $J$ as a linear map $V_2 \to V_1$, we may consider its operator norm $N := \|B\|$ with respect to the scalar products $g_2$ and $g_1$. Working in a basis as in the previous paragraph and using \ref{eq:system}, we compute that 
\begin{equation}\label{eq:form of B}
B = \begin{pmatrix} a & b \\ b & -a \end{pmatrix} ,
\end{equation}
from which it is easily seen that $N = \sqrt{a^2+b^2}$. Although the matrix $B$ is determined only up to left and right actions of $\SO(2)$, its norm $N$ is independent of the specific basis picked: it is an intrinsic invariant of the pair $(\omega,J)$. As we shall explain below, $N$ measures how far the tame pair $(\omega, J)$ is from being compatible.

%%%%%%%%%%%%%%%%%%%%%%%%%%%%%%%%%%%%%%%%%%%%%%%%%%%%%%%%%%%%%%%%%%%%%%%%%%%%%%%%
\subsection{Tameness and compatibility conditions}\label{section:tameness+compatibility conditions}
%%%%%%%%%%%%%%%%%%%%%%%%%%%%%%%%%%%%%%%%%%%%%%%%%%%%%%%%%%%%%%%%%%%%%%%%%%%%%%%%
%
Given a matrix $B$ of the form \eqref{eq:form of B}, let $J_B$ be the endomorphism represented by the block matrix
\[ \begin{pmatrix} J_0 & B \\ 0 & J_0\end{pmatrix}.\]
We note that 
\[ B^T = B,~\; J_0B = -BJ_0 =\begin{pmatrix} -b & a \\ a & b  \end{pmatrix},~\; B^2 = (a^2+b^2)I,\; \text{~and that~}\; J_B^2 = -I.\]
\begin{lemma}\label{lemma:tameness condition}
The linear pair $(\om, J_B)$ is tame if and only if $N :=\|B\| < 2$, and it is compatible if and only if $N = 0$.
\end{lemma}
\begin{proof}
Given $(v,w) \in V_1 \oplus V_2$ of norm $1$, we have
\begin{align}\label{eq:tameness_N}
(v^T \; w^T) \om J_B \left( \begin{array}{c} v \\ w \end{array} \right)& = (v^TJ_0^T \;\; w^TJ_0^T)\left( \begin{array}{c} J_0v + Bw \\ J_0w \end{array} \right) = \|J_0 v\|^2 + \|J_0 w\|^2 - v^TJ_0B w \\
\notag &= 1 - v^TJ_0Bw \ge 1 - \|J_0 B\| \|v\| \|w\| \ge 1 - N/2 
\end{align}
where equalities hold when $v$ and $J_0Bw$ are colinear and $\|v\| = \|w\| = 1/\sqrt{2}$. The pair $(\om,J_B)$ is tame if and only if the above lower bound is positive, namely if and only if $N < 2$. Furthermore, the pair $(\om, J_B)$ is $J_B$-invariant if and only if $J_B^T \omega J_B = \omega$, which holds (as a straightforward computation of the matrix product shows) if and only if $B = 0$ or, equivalently,  $N = 0$.
\end{proof}
For later use, we also note that
\begin{align}
 \| J_B \| &:= \underset{\|v\|^2 + \|w\|^2 = 1}{\max} \; \left\| J_B \left(\begin{array}{c} v \\ w \end{array} \right) \right\| = \underset{\|v\|^2 + \|w\|^2 = 1}{\max} \; \sqrt{\|J_0v + Bw\|^2 + \|J_0 w\|^2} \\
\notag &\le \underset{\|v\|^2 + \|w\|^2 = 1}{\max} \; \sqrt{\|v\|^2 + \|w\|^2 + 2 \|v\|\|w\| N + N^2 \|w\|^2 }  \le \sqrt{1+N+N^2} \le 1 + N \; .
\end{align}

%%%%%%%%%%%%%%%%%%%%%%%%%%%%%%%%%%%%%%%%%%%%%%%%%%%%%%%%%%%%%%%%%%%%%%%%%%%%%%%%
\section{Tame-to-tame inflation via Thom forms}\label{section:Thom forms}
%%%%%%%%%%%%%%%%%%%%%%%%%%%%%%%%%%%%%%%%%%%%%%%%%%%%%%%%%%%%%%%%%%%%%%%%%%%%%%%%
In this section, we prove the tame-to-tame inflation lemma in the cases $Z\cdot Z=0$ and $Z\cdot Z<0$ following closely the original arguments given in~\cite{McDuff-J-inflation} and~\cite{Buse-Neg-inflation}. We also point out why this approach fails for surfaces with $Z\cdot Z>0$.

%%%%%%%%%%%%%%%%%%%%%%%%%%%%%%%%%%%%%%%%%%%%%%%%%%%%%%%%%%%%%%%%%%%%%%%%%%%%%%%%
\subsection{Case of surfaces with trivial normal bundles}\label{Subsection:TrivialBundle}
%%%%%%%%%%%%%%%%%%%%%%%%%%%%%%%%%%%%%%%%%%%%%%%%%%%%%%%%%%%%%%%%%%%%%%%%%%%%%%%%
After rescaling, we can suppose $\om(Z)=1$. By Weinstein's symplectic neighborhood theorem, a neighborhood of $Z$ can be symplectically identified with the product $Z\times D^{2}(r_0)$, equipped with the product symplectic form $\om_0=\sigma+ dx\wedge dy = \sigma+ rdr\wedge d\theta$, where $D^{2}(r_0)\subset\R^2$ is a standard disc of some radius $r_0>0$, and where $\sigma=\om_0|_{TZ}$. Let $J$ be a tame almost complex structure for which $Z=Z\times \{0\}$ is $J$-holomorphic. Given any $p\in Z$, we can find local symplectic coordinates $(x_1,y_1,x_2,y_2)$ that are unitary at $p$ and for which $Z$ is given by $x_2=y_2=0$. We write tangent vectors as pairs $(u,v)$ in which $u$ is horizontal and $v$ is vertical. In these coordinates $\om_0$ and $J$ are represented by the block matrices
\[
\om_0=\begin{pmatrix}J_0 & 0\\ 0 & J_0\end{pmatrix},~\quad\quad~J=\begin{pmatrix}A & B\\ C & D\end{pmatrix}.
\]
It follows that near $p\in Z$,
\begin{align*}
0<g_J\big((u,v),\, (u,v)\big)=||(u,v)||_{J}^{2}&:=\om_0\big((u,v),\, J(u,v)\big)\\
&\phantom{:}= u^{\trans}J_{0}^{\trans}Au + u^{\trans}J_{0}^{\trans}Bv
+ v^{\trans}J_{0}^{\trans}Cu + v^{\trans}J_{0}^{\trans}Dv.
\end{align*}
Note that at $p$ itself, the matrix $J$ takes the form
\begin{equation}\label{eq:block matrix expression for J along Z}
J_p=\begin{pmatrix}J_{0} & B_{0}\\ 0 & J_{0}\end{pmatrix}\; 
\text{~where~}\;
B_{0}=\begin{pmatrix}a&b\\b&-a\end{pmatrix}.
\end{equation}
Let $g$ be the scalar product on $T_Z M$ defined as in section~\ref{sec:tion:canonical_metric}. Because the matrices $A$ and $D$ are nondegenerate in a neighborhood of $p\in Z$, we can extend $g$ in a neighborhood of $p$ by setting
\begin{align*}
g\big((u,v),\, (u',v')\big)&:= u^{\trans}J_{0}^{\trans}Au' + v^{\trans}J_{0}^{\trans}Dv'\\
&\phantom{:}=g_A\big(u,\, u'\big)+ g_D\big(v,\, v'\big).
\end{align*}
Let $\|\cdot\|_g$ be the associated norm.

We observe that, in some possibly smaller neighborhood $Z\times D(r_1)$, $r_{1}\leq r_{0}$, the tameness of $J$ implies that $||J_{0}^{\trans}B||_{g}<2$, so that
\begin{equation}\label{eq:epsilon_1 in terms of |B|}
|u^{\trans}J_{0}^{\trans}Bv|\leq ||J_{0}^{\trans}B||_{g}\cdot||u||_{A}\cdot||v||_{D}\leq\frac{1}{2}||J_{0}^{\trans}B||_{g}\cdot\Big(||u||_{A}^{2}+||v||_{D}^{2}\Big)
\leq \epsilon_{1}\Big(||u||_{A}^{2}+||v||_{D}^{2}\Big)
\end{equation}
where $0<\epsilon_{1}<1$. (The middle inequality is just the rearrangement inequality applied to the positive numbers $||u||_A$ and $||v||_D$.)  
Near $Z$, the same ideas apply to the matrix $J_{0}^{\trans}C$, but because $C(0)=0$, we can factor $r$ out of $||J_{0}^{\trans}C||_{g}$ so that, in some possibly smaller neighborhood $Z\times D(r_2)$, $r_{2}\leq r_{1}$, 
\begin{equation}\label{eq:epsilon_2 in terms of |C|}
|u^{\trans}J_{0}^{\trans}Cv|\leq ||J_{0}^{\trans}C||_{g}\cdot||u||_{A}\cdot||v||_{D}\leq\frac{1}{2}r\epsilon_{2}\cdot\Big(||u||_{A}^{2}+||v||_{D}^{2}\Big)
\end{equation}
for some constant $\epsilon_{2}>0$. So, in some neighborhood of $p\in Z$, we have
\[ \big|u^{\trans}J_{0}^{\trans}Bv \big| \leq \epsilon_{1}\Big(u^{\trans}J_{0}^{\trans}Au + v^{\trans}J_{0}^{\trans}Dv\Big)
\quad\text{~and~}\quad
\big|v^{\trans}J_{0}^{\trans}Cu \big| \leq r\epsilon_{2}\Big(u^{\trans}J_{0}^{\trans}Au + v^{\trans}J_{0}^{\trans}Dv\Big)\]
for some constants $0<\epsilon_{1}<1$ and $0<\epsilon_{2}$.

Given $t\geq 0$, we want to define a taming symplectic form $\om'$ in class $[\om]+t\PD(Z)$. Note that the Thom class of $Z$ can be represented by the form $\rho(r)rdr\wedge d\theta$, where $\rho:[0,R]\to\R_+$ is a non-increasing cutoff function that is constant near $0$, that vanishes near $R$, and such that $\int_{D(R)} \rho(r)rdr d\theta = 1$. Set $\om_{f}=\sigma+ f(r) rdr\wedge d\theta$ for some $f(r)\geq 1$. Near $p\in Z$ we can write
\[
\om_{f}((u,v),J(u,v))
= u^{\trans}J_{0}^{\trans}Au + u^{\trans}J_{0}^{\trans}Bv
+fv^{\trans}J_{0}^{\trans}Cu + fv^{\trans}J_{0}^{\trans}Dv
\]
and
\begin{align}
\Big| u^{\trans}J_{0}^{\trans}Bv+fv^{\trans}J_{0}^{\trans}Cu \Big|
&\leq \Big| u^{\trans}J_{0}^{\trans}Bv\Big|+\sqrt{f}\Big|\left(\sqrt{f}v^{\trans}\right)J_{0}^{\trans}Cu \Big| \label{eq:inequality norm one}\\
&\leq \epsilon_{1}\Big(u^{\trans}J_{0}^{\trans}Au + v^{\trans}J_{0}^{\trans}Dv\Big) + \epsilon_{2}r\sqrt{f}\Big(u^{\trans}J_{0}^{\trans}Au + fv^{\trans}J_{0}^{\trans}Dv\Big)\label{eq:inequality norm two}\\
&\leq \epsilon_{1}\Big(u^{\trans}J_{0}^{\trans}Au + fv^{\trans}J_{0}^{\trans}Dv\Big) + \epsilon_{2}r\sqrt{f}\Big(u^{\trans}J_{0}^{\trans}Au + fv^{\trans}J_{0}^{\trans}Dv\Big)\nonumber\\
&= \Big(\epsilon_{1}+\epsilon_{2}r\sqrt{f}\Big) \Big(u^{\trans}J_{0}^{\trans}Au + fv^{\trans}J_{0}^{\trans}Dv\Big)\nonumber
\end{align}
It follows that $\om_f$ tames $J$ near $p$ whenever $\Big(\epsilon_{1}+\epsilon_{2}r\sqrt{f}\Big)<1$. Consequently, in order to prove the tame-to-tame inflation lemma for $Z\cdot Z=0$, we have to show that given $t\geq 0$ arbitrarily large, we can find a function $f\geq 1$ and a radius $0<R\leq r_2$ such that: (i) $f=1$ outside $D(R)$, (ii) $\epsilon_{1}+\epsilon_{2}r\sqrt{f}<1$ in $D(R)$, and (iii) the integral of $f-1$ over $D(R)$ is $t$. Now, the inequality $\epsilon_{1}+\epsilon_{2}r\sqrt{f}<1$ is equivalent to 
\[1\leq f(r) < \frac{(1-\epsilon_{1})^{2}}{r^{2}\epsilon_{2}^{2}}\]
which can be achieved by taking $R$ small enough. Moreover, since
\[\int_{D(R)} \frac{rdr d\theta}{r^{2}}\]
is unbounded, we can make the integral of $f$ on $D(R)$ as large as we want. Repeating the same argument for finitely many charts covering $Z$, this shows that we can find some $R>0$ and some function $f:D(R)\to\R$ such that $\om_f$ tames $J$ and $[\om_f]=[\om]+t\PD(Z)$.

% Case of a negative normal bundle 
%%%%%%%%%%%%%%%%%%%%%%%%%%%%%%%%%%%%%%%%%%%%%%%%%%%%%%%%%%%%%%%%%%%%%%%%%%%%%%%%
\subsection{Case of surfaces with negative normal bundles}\label{subsection:Negativebundle}
%%%%%%%%%%%%%%%%%%%%%%%%%%%%%%%%%%%%%%%%%%%%%%%%%%%%%%%%%%%%%%%%%%%%%%%%%%%%%%%%
The case of a surface of negative self-intersection is similar but requires a different local model. Suppose that $Z\subset M$ is $J$-holomorphic and that $Z\cdot Z= -m<0$. We can rescale $\om$ so that $\om(Z)=1$. A neighborhood of $Z$ can be identified with a tubular neighborhood of the zero section $\mathbf{0}$ of the symplectic normal bundle $\pi:N\to Z$. Let $P\subset N$ be the unit sphere bundle seen as a principal $U(1)$-bundle, and let $\beta$ be a connection one-form on $P$ such that $d\beta=m\pi^{*}\sigma$. Let $\Pi:N\setminus \mathbf{0}\to P$ be the Gauss map, and define the $1$-form $\alpha=\Pi^*\beta$ on $N\setminus \mathbf{0}$. Note that $\alpha$ has a pole of order $1$ along $Z$ while $r^{2}\alpha$ is everywhere smooth. Weinstein's theorem implies that $\om$ is isotopic to the standard form
\[\om_{0}=\pi^{*}(\sigma) + \frac{1}{2}d(r^{2}\alpha)\]
near $Z$. We define 
\[\om_{f}:= \pi^{*}(\sigma) + \frac{1}{2}d(r^{2}\alpha) - d(f(r)\alpha) = (1+\frac{1}{2}mr^{2}-mf)\pi^{*}\sigma + \left(1-\frac{f'}{r}\right)rdr\wedge\alpha\]
for some non-increasing function $f(r)\geq 0$ with support in a tubular neighborhood of $Z$. By Stokes' theorem, the cohomology class of $\om_{f}$ is $[\om]+f(0)\PD(Z)$. In particular, for such a form to tame $J$, we must have $0\leq f(r)< 1/m$ on $Z$, and if $0<M'<1/m$ is the inflation parameter we want to reach, we must have $f(0)=M'$.

Pick $p\in Z$ and choose a local symplectic frame compatible with the horizontal and vertical splitting given by the connection. With respect to such a frame, we have
\[
\om_{0}\big((u,v),(u',v')\big) = u^{\trans}J_{0}^{\trans}u' + v^{\trans}J_{0}^{\trans}v'
\]
and
\[
\om_{f}\big((u,v),(u',v')\big) = au^{\trans}J_{0}^{\trans}u' + bv^{\trans}J_{0}^{\trans}v'
\]
where 
\[a:=1-\frac{mf}{1+\frac{1}{2}mr^{2}} \text{~and~} b:=1-\frac{f'}{r}.\]
Since we suppose $0\leq f(r)<1/m$ is non-increasing, we have
\[0<a<1\text{~and~}1\leq b\]
where $a=a(r)$ is non-decreasing and bounded below by $a_0:=a(0)=1-mf(0)=1-mM'>0$. We can now write
\begin{align*}
\om_{f}\big((u,v),J(u,v)\big)
&= a\Big(u^{\trans}J_{0}^{\trans}Au + u^{\trans}J_{0}^{\trans}Bv\Big)
+ b\Big(v^{\trans}J_{0}^{\trans}Cu + v^{\trans}J_{0}^{\trans}Dv\Big)\\
&\hspace*{-6mm}=\om_{0}\big((u,v),J(u.v)\big) - \frac{mf}{1+\frac{1}{2}mr^{2}}\Big(u^{\trans}J_{0}^{\trans}Au + u^{\trans}J_{0}^{\trans}Bv\Big)
- \frac{f'}{r}\Big(v^{\trans}J_{0}^{\trans}Cu + v^{\trans}J_{0}^{\trans}Dv\Big).
\end{align*}
Given any $0<\epsilon_1<1$ and $R'>0$ such that
\begin{equation}\label{eq:B inequality}
\big|u^{\trans}J_{0}^{\trans}Bv \big| \leq \epsilon_1\Big(u^{\trans}J_{0}^{\trans}Au + v^{\trans}J_{0}^{\trans}Dv\Big)
\end{equation}
for all $r\leq R'$, we have
\begin{align}\label{eq:1st inequality}
a\big|u^{\trans}J_{0}^{\trans}Bv \big| &\leq \epsilon_{1}\sqrt{a}\Big(au^{\trans}J_{0}^{\trans}Au + v^{\trans}J_{0}^{\trans}Dv\Big)\\
&\leq \epsilon_{1}\sqrt{a}\Big(au^{\trans}J_{0}^{\trans}Au + bv^{\trans}J_{0}^{\trans}Dv\Big).
\end{align}
Similarly, if $\epsilon>0$ is such that, for $r\leq R'$, we have 
\[
\big|v^{\trans}J_{0}^{\trans}Cu \big| \leq r\epsilon\Big(u^{\trans}J_{0}^{\trans}Au + v^{\trans}J_{0}^{\trans}Dv\Big)
\]
then we can set $\epsilon_{2} = \epsilon/\sqrt{a_0}$ to get
\begin{equation}\label{eq:2nd inequality}
b\big|v^{\trans}J_{0}^{\trans}Cu \big| \leq r\epsilon_{2}\sqrt{b}\Big(au^{\trans}J_{0}^{\trans}Au + bv^{\trans}J_{0}^{\trans}Dv\Big).
\end{equation}
Note that for any fixed inflation parameter $0<M'<1/m$, both constants $\epsilon_1$ and $\epsilon_2$ are independent of $f$ and only depend on $J$, $R'$, and on the chosen symplectic frame. We then have
\begin{align*}
\big|au^{\trans}J_{0}^{\trans}Bv + bv^{\trans}J_{0}^{\trans}Cu\big|
&\leq a\big|u^{\trans}J_{0}^{\trans}Bv\big|+b\big|v^{\trans}J_{0}^{\trans}Cu\big|\\
&\leq \big(\epsilon_{1}\sqrt{a}+r\epsilon_{2}\sqrt{b}\big)\big(au^{\trans}J_{0}^{\trans}Au +bv^{\trans}J_{0}^{\trans}Dv\big).
\end{align*}
To ensure tameness of $J$, we want $\big(\epsilon_{1}\sqrt{a}+r\epsilon_{2}\sqrt{b}\big)<1$. Since we always have $\epsilon_{1}\sqrt{a}<1$, we can rewrite this as
\[
r\epsilon_{2}\sqrt{b}<1-\epsilon_{1}\sqrt{a}.
\]
Recall that the function $f$ must be supported in an arbitrarily small neighborhood of $Z$, must be non-increasing, and must be equal to $M'$ near $Z$. We will take $f$ to be a smoothing of a function $h$ of the form
\[
h(r)=
\begin{cases}
M' & 0\leq r \leq R_1\\
\left(\frac{M'}{\log(R_2/R_1)}\right)\cdot\big(\log(R_2)-\log(r)\big) & R_1\leq r\leq R_2\\
0 & R_2\leq r
\end{cases}
\]
for some constants $0<R_1<R_2<R<R'$. We can assume that the smoothing $f$ is non-increasing and is supported in $[0,R]$. Writing $c=\frac{M'}{\log(R_2/R_1)}$, it follows that $-f'(r)\leq c/r$, so that
\[
r\epsilon_{2}\sqrt{b}=r\epsilon_{2}\sqrt{1-f'/r}=\epsilon_{2}\sqrt{r^{2}-rf'}\leq\epsilon_{2}\sqrt{r^{2}+c}.
\]
Observe that we can always choose $0<R_1<R_2<R$ with $R$ small enough such that, for all $r<R$, we have
\[
\epsilon_{2}\sqrt{r^{2}+c}=\epsilon_2\sqrt{r^2+\frac{M'}{\log(R_2/R_1)}}<1-\epsilon_{1}<1-\epsilon_{1}\sqrt{a}.
\] 
This shows that a suitable function $f(r)$ can be locally constructed near every point $p\in Z$. Since $f(r)$ depends only on $r$, we can repeat the same argument in finitely many charts covering $Z$ to find a suitable function defined in a whole tubular neighborhood of~$Z$.

% Case of a positive normal bundle 
%%%%%%%%%%%%%%%%%%%%%%%%%%%%%%%%%%%%%%%%%%%%%%%%%%%%%%%%%%%%%%%%%%%%%%%%%%%%%%%%
\subsection{Case of surfaces with positive normal bundles}
%%%%%%%%%%%%%%%%%%%%%%%%%%%%%%%%%%%%%%%%%%%%%%%%%%%%%%%%%%%%%%%%%%%%%%%%%%%%%%%%
We now explain the appearance of an upper bound $T\sim C(J)\cdot\frac{1}{m}$ on the inflation parameter $t$ when we apply the previous arguments in the case of a surface of strictly positive self-intersection. Suppose that $Z\subset M$ is $J$-holomorphic and that $Z\cdot Z= m>0$. As before, we work in a  neighborhood of the zero section in the symplectic normal bundle endowed with the form
\[\om_{0}=\pi^{*}(\sigma) + \frac{1}{2}d(r^{2}\alpha)\]
where $\sigma=\om|_Z$ and where $\alpha$ is obtained from a connection one-form $\beta$ such that $d\beta=-m\pi^{*}\sigma$. Given $M'>0$, we define 
\[\om_{f}:= \pi^{*}(\sigma) + \frac{1}{2}d(r^{2}\alpha) - d(f(r)\alpha) = (1-\frac{1}{2}mr^{2}+mf)\pi^{*}\sigma + \left(r-\frac{f'}{r}\right)rdr\wedge\alpha\]
for some suitable non-increasing function $f(r)\geq 0$ which takes the value $M'$ near $Z$, and with support in an arbitrarily small neighborhood of $Z$. 

In a local symplectic frame compatible with the horizontal splitting given by the connection, we have
\[\om_{0}\big((u,v),(u',v')\big) = u^{\trans}J_{0}^{\trans}u' + v^{\trans}J_{0}^{\trans}v'\]
while 
\[\om_{f}\big((u,v),(u'.v')\big) = au^{\trans}J_{0}^{\trans}u' + bv^{\trans}J_{0}^{\trans}v'\]
where $a$ and $b$ are functions of $r$ given by 
\[a=1+\frac{mf}{1-\frac{1}{2}mr^{2}} \quad\text{~and~}\quad b=1-\frac{f'}{r}\geq 1.\]
This time, for $r<1/\sqrt{m}$, we have
\[a \geq 1 \quad\text{~and~}\quad b\geq 1.\]
As before,
\begin{align*}
\om_{f}\big((u,v),J(u,v)\big)
&= a\Big(u^{\trans}J_{0}^{\trans}Au + u^{\trans}J_{0}^{\trans}Bv\Big)
+ b\Big(v^{\trans}J_{0}^{\trans}Cu + v^{\trans}J_{0}^{\trans}Dv\Big)\\
\end{align*}
Given $0<\epsilon_1<1$ and $0<\epsilon_2$ such that 
\begin{equation}\label{eq:B inequality2}
\big|u^{\trans}J_{0}^{\trans}Bv \big| \leq \epsilon_1\Big(u^{\trans}J_{0}^{\trans}Au + v^{\trans}J_{0}^{\trans}Dv\Big)
\end{equation}
and 
\[
\big|v^{\trans}J_{0}^{\trans}Cu \big| \leq r\epsilon_2\Big(u^{\trans}J_{0}^{\trans}Au + v^{\trans}J_{0}^{\trans}Dv\Big)
\]
we have
\begin{align*}
\big|au^{\trans}J_{0}^{\trans}Bv + bv^{\trans}J_{0}^{\trans}Cu\big|
&\leq a\big|u^{\trans}J_{0}^{\trans}Bv\big|+b\big|v^{\trans}J_{0}^{\trans}Cu\big|\\
&\leq \epsilon_{1}\sqrt{a}\big(au^{\trans}J_{0}^{\trans}Au +v^{\trans}J_{0}^{\trans}Dv\big) + r\epsilon_{2}\sqrt{b}\big(u^{\trans}J_{0}^{\trans}Au +bv^{\trans}J_{0}^{\trans}Dv\big)\\
&\leq \big(\epsilon_{1}\sqrt{a}+r\epsilon_{2}\sqrt{b}\big)\big(au^{\trans}J_{0}^{\trans}Au +bv^{\trans}J_{0}^{\trans}Dv\big)
\end{align*}
We want $\big(\epsilon_{1}\sqrt{a}+r\epsilon_{2}\sqrt{b}\big)<1$. In particular, for $r$ very close to $0$, we have $f=M'$ and $f'=0$, which implies $a\geq 1+mM'$ and $b=1$. Consequently,
\[
M'=f(r)<\left(\frac{1-\epsilon_{1}^{2}}{\epsilon_{1}^{2}}\right)\left(\frac{1-\frac{1}{2}mr^{2}}{m}\right)< \left(\frac{1-\epsilon_{1}^{2}}{\epsilon_{1}^{2}}\right)\frac{1}{m}
\]
Looking back at equation~\eqref{eq:epsilon_1 in terms of |B|}, we see that $\frac{1}{2}\|B\|_g<\epsilon_1<1$ near $p\in Z$. Therefore the last inequalities impose an upper bound on $M'$ of the form
\[M'<C(J)\cdot\frac{1}{m}\] 
where $C(J)$ approaches $0$ as $N(J):=\|B\|_g$ approaches its supremum $2$.

\begin{remark}
In~\cite{McDuff-J-inflation} and~\cite{Buse-Neg-inflation}, the block-matrix expression for $J$ given in equation~\eqref{eq:block matrix expression for J along Z} is assumed to satisfy $B=C=0$ along $Z$, which is only possible when the symplectic normal bundle $\nu_Z$ is $J$-invariant. In turns, this implies that $J$ is compatible with $\om$ along $Z$. Under this compatibility assumption, the inequality 
\[ \big|u^{\trans}J_{0}^{\trans}Bv \big| \leq \epsilon_{1}\Big(u^{\trans}J_{0}^{\trans}Au + v^{\trans}J_{0}^{\trans}Dv\Big)\]
that is used above to go from equation~\eqref{eq:inequality norm one} to equation~\eqref{eq:inequality norm two} can be improved to
\[ \big|u^{\trans}J_{0}^{\trans}Bv \big| \leq r\epsilon_{1}'\Big(u^{\trans}J_{0}^{\trans}Au + v^{\trans}J_{0}^{\trans}Dv\Big)\]
for some $\epsilon_{1}'>0$. The extra $r$ factor is what makes the slightly simpler arguments given in~\cite{McDuff-J-inflation} and~\cite{Buse-Neg-inflation} work in the compatible case even if $Z$ is of positive self-intersection.\eoe
\end{remark}

%%%%%%%%%%%%%%%%%%%%%%%%%%%%%%%%%%%%%%%%%%%%%%%%%%%%%%%%%%%%%%%%%%%%%%%%%%%%%%%%
\section{Initial considerations towards the Preparation Lemma}\label{section:towards the preparation lemma}
%%%%%%%%%%%%%%%%%%%%%%%%%%%%%%%%%%%%%%%%%%%%%%%%%%%%%%%%%%%%%%%%%%%%%%%%%%%%%%%%

Our strategy to prove the Preparation lemma is to keep the symplectic form $\om$ fixed while deforming the tamed almost complex structure $J$ through an isotopy obtained from local diffeotopies. We shall deduce the Preparation lemma from the following statement\footnote{In what follows, given a diffeomorphism $\psi$ of $M$ and a bundle morphism $J$ of $TM$, we shall write $\psi^*J$ for the bundle morphism $(\psi^*J)_p(v) = [(T_p \psi)^{-1} \circ J_{\psi(p)} \circ (T_p \psi)](v)$ where $v \in T_p M$.}:

\begin{theorem}\label{thm2}
There is a diffeomorphism $\psi \in\Diff(M,~ \id\text{~on~}Z)\cap\Diff_0(M)$ i.e. which fixes $Z$ pointwise and is isotopic to the identity, such that $\psi^*J$ is compatible with $\omega$ along $Z$ and is $\om$-tame everywhere. Moreover, $\psi$ can be chosen $C^0$-small and supported in any open neighborhood of $Z$.
\end{theorem}

\noindent The preparation lemma readily follows from this Theorem: denoting by $\{\psi_t\}_{0 \le t \le 1}$ the isotopy between $\psi_0 = id_M$ and $\psi_1 = \psi$, we get a path $\omega_t = (\psi_t^{-1})^* \omega$ of cohomologous closed $2$-forms such that $\omega_0 = \omega$ and $\omega_1 =: \omega'$, with $\omega'$ taming $J$ over the whole of $M$ and compatible with it along $Z$. 

\begin{remark} In fact, the above isotopy $\{\psi_t\}_{0 \le t \le 1}$ can be chosen so that $Z$ is fixed pointwise and $J_t = \psi_t^* J$ is $\omega$-tamed for all $0 \le t \le 1$. A first way to see this is by slightly refining our proof of Theorem \ref{thm2}, using Remark \ref{rem:isotopy_2} to define an appropriate isotopy and applying Lemma \ref{lem:stability} to it. Another way is to start with the $\psi$ given by Theorem \ref{thm2} and apply Moser's path argument to change the isotopy: more precisely,  the symplectic form $\omega' := (\psi^{-1})^*\omega$ is compatible with $J$ along $Z$ and cohomologous to $\omega$. Hence the path $\omega_s := (1-s)\omega + s \omega'$ ($s \in [0,1]$) consists of cohomologous forms that are all symplectic, as they all tame $J$, and which all coincide with $\omega$ on $TZ$. By a refinement of Moser's trick (cf. Remark \ref{rem:moser_trick}), this path between $\omega$ and $\omega'$ is induced by pullbacks along a diffeotopy of $M$ fixing $Z$ pointwise.\eoe
\end{remark}\vspace{0pt}

\subsection{Overview of the argument} 
Let $T_Z M$ denote the tangent bundle of $M$ restricted to $Z$. It is relatively easy to construct an isotopy of automorphisms $\Psi_t : T_Z M \to T_Z M$ (lifting the identity $id : Z \to Z$), $t\in[0,s]$, such that $\Psi_t^*J$ is $\omega$-tame and $\Psi_s^*J$ is compatible with $\omega$; see Section~\ref{section:linear isotopy} below. Due to the large flexibility present within the category of smooth manifolds, it seems reasonable to expect $\Psi_s$ to be the restriction to $T_Z M$ of the differential of a diffeomorphism $\psi$ of $M$. This easily follows from the ‘Whitney extension-type’ result proved in Section~\ref{section:Whitney_ext}. Whether such an extension $\psi$ of $\Psi_s$ can be found such that $\omega$ tames $\psi^*J$ everywhere makes the problem much more delicate. The purpose of Section~\ref{sec:stability} is to obtain a quantitative control over the $\omega$-tameness of almost complex structures under the action of sufficiently $C^1$-small diffeomorphisms with given supports. Granted this, our strategy consists in finding a suitable constant $\epsilon>0$ and a partition $t_0=0<t_1<\cdots<t_d=s$ of the interval $[0,s]$ such that the $C^1$-norm of $\psi_{t_i}\psi_{t_{i-1}}^{-1}$ is bounded by $\epsilon$. Starting with $t_1$, we iteratively find extensions of the automorphisms $\Psi_{t_i}\Psi_{t_{i-1}}^{-1}$ whose supports form a nested sequence of tubular neighborhoods chosen in such a way that the control over the $C^1$ norm suffices to ensure tameness of the pull-back of $J$ everywhere on $M$. Roughly speaking, finding such a nested sequence of supports is possible as the "degree of compatibility" of $\om$ and $J$ improves near $Z$ after each step. The details are provided in Section~\ref{section:prep_lemma}.

%%%%%%%%%%%%%%%%%%%%%%%%%%%%%%%%%%%%%%%%%%%%%%%%%%%%%%%%%%%%%%%%%%%%%%%%%%%%%%%%
\subsection{Isotopy of the tangent bundle \texorpdfstring{$T_ZM$}{TZM}}\label{section:linear isotopy}
%%%%%%%%%%%%%%%%%%%%%%%%%%%%%%%%%%%%%%%%%%%%%%%%%%%%%%%%%%%%%%%%%%%%%%%%%%%%%%%%

Our first step in the proof of Theorem \ref{thm2}, as done in section~\ref{subsubsec:setup_lin_level}, consists in establishing the existence of an appropriate isotopy of fiberwise linear automorphisms of $T_Z M$. This motivates the considerations of this section.\vspace{0pt}

The $\om_t$-orthogonal to $V_1 = \R^2 \oplus 0$ is the subspace $W_t := \{ (tJ_0Bv, v) \, | \, v \in \R^2 \}$. For $t \in [0,1]$, we define the map
\[ L_t : V_2 = 0 \oplus \R^2 \to V_1 = \R^2 \oplus 0 : v \mapsto tJ_0Bv \, .\]
From them, we construct the following diffeotopy of $V = V_1 \oplus V_2 =  \R^2 \oplus \R^2$:
\[ \Psi_t : \R^2 \oplus \R^2 \to \R^2 \oplus \R^2 : (u,v) \mapsto (u+ \alpha(t) L_t(v), \alpha(t)v) \qquad (t \in \R)\]
where $\alpha(t) := (1 - N^2 \, t(1-t))^{-1/2}$. (This is well-defined precisely because $N < 2$.) Each map $\Psi_t$ is indeed invertible, with inverse 
\[ \Psi_t^{-1}(u,v) = (u - L_t(v), \alpha(t)^{-1}v) \, .\]
By similar arguments as in Section \ref{section:tameness+compatibility conditions}, we get the estimate
\begin{align}\label{eq:norme_psi-id}
 \| \Psi_t - Id \| &:= \underset{\|u\|^2 + \|v\|^2 = 1}{\max} \; \left\| \Psi_t(u,v) - (u,v) \right\| \le \left((\alpha(t)-1)^2 + (t N \alpha(t))^2\right)^{1/2} \, .
\end{align}
 Straightforward calculations also yield
\begin{align} \label{eq:matrixform4} \Psi_t^*\om_t : = \Psi_t^T \om_t \Psi_t = \om \; \text{~and~} \; \Psi_t^*J := \Psi_t^{-1}J\Psi_t =  \begin{pmatrix} J_0 & (1-2t)\alpha(t) B \\ 0 & J_0  \end{pmatrix} \, .\end{align}

We stress that all the objects we just defined have basis-free descriptions, and are thus canonically associated with the pair $(\omega, J)$. 

Consider  $N(t) := \| (1-2t)\alpha(t)B\| = |1-2t| \alpha(t) N $, where the norm is still computed with respect to the standard scalar product on $\R^4$. We note that $t \mapsto N(t)$ is symmetric about $t = 1/2$, is decreasing over $[0, 1/2]$, and satisfies $N(0) = N$ and  $N(1/2) = 0$. Consequently and since $N < 2$, $N(t) < 2$ for all $t \in [0,1]$, which means that all pairs $(\om, \Psi_t^* J)$ -- equivalently, all pairs $(\om_t, J)$ -- are tame for $t \in [0,1]$.

In view of Equation \ref{eq:matrixform4}, all complex structures $\Psi_t^* J$ determine the same complex structures on $V_1$ and $V_2$, namely $J_1$ and $J_2$. Accordingly, from the discussion in Section \ref{sec:tion:canonical_metric}, for each $t \in [0,1]$, the canonical metric associated with the tame pair $(\om, \Psi_t^* J)$  is equal to the canonical metric $g$ associated with $(\om, J)$. It hence follows that for each $t \in [0,1]$, $N(t)$ is the norm of the skew part of $\Psi_t^* J$.

In conclusion, given the tame pair $(\om, J)$, the corresponding pairs $(\om, \Psi_t^*J)$  are tame for all $t \in [0,1]$ and they all define the same canonical structures $(V_1, V_2, J_1, J_2, g)$ on $(V, \omega)$. Moreover, from $N(1/2) = 0$, we see that the pair $(\omega, \Psi_{1/2}^*J)$ is compatible. For the purpose of the preparation lemma, we only need to focus on the time-interval $t \in [0,1/2]$.

%%%%%%%%%%%%%%%%%%%%%%%%%%%%%%%%%%%%%%%%%%%%%%%%%%%%%%%%%%%%%%%%%%%%%%%%%%%%%%%%
\subsection{A technical lemma} 
%%%%%%%%%%%%%%%%%%%%%%%%%%%%%%%%%%%%%%%%%%%%%%%%%%%%%%%%%%%%%%%%%%%%%%%%%%%%%%%%
The following estimate will be used in subsections~\ref{subsubsec:epsilon} and~\ref{subsec:Choice_of_partition}.

\begin{lemma}[Estimates on $|\Psi_t|$]\label{lem:newtechnical}
Using notation from the setup in Section \ref{section:preliminary considerations}: let $N = \| B\| < 2$.  For $\epsilon > 0$, if $0 \le t' < t \le 1/2$ satisfy $t - t' < \dfrac{\epsilon}{\sqrt{2}} \left( \dfrac{1}{N} - \dfrac{1}{2} \right)$, then $\| \Psi_t \circ \Psi_{t'}^{-1} - Id \| < \epsilon$.
\end{lemma}

\begin{proof} 
If $N=0$, then $\Psi_t = \Psi_{t'}$ for all $t, t' \in [0,1/2]$ and the claim is clear. Hence we consider $N > 0$. For convenience, set $C := \epsilon/\sqrt{2}$.

First, we note that
\[ ( \Psi_t \circ \Psi_{t'}^{-1})(u,v) = \left(  u - L_{t'}(v) + \dfrac{\alpha(t)}{\alpha(t')} L_t(v)  \, , \, \dfrac{\alpha(t)}{\alpha(t')} v \right) \, .  \]
Hence
\begin{flalign*}
(\Psi_t \circ \Psi_{t'}^{-1})(u,v) - (u,v) &= \left(  - L_{t'}(v) + \dfrac{\alpha(t)}{\alpha(t')} L_t(v)  \, , \, \left( \dfrac{\alpha(t)}{\alpha(t')} - 1 \right) v \right) \\
&= \left( \left( \dfrac{\alpha(t)}{\alpha(t')}t - t' \right) J_0 B v  \, , \, \left( \dfrac{\alpha(t)}{\alpha(t')} - 1 \right) v \right) \, .
\end{flalign*}

Thus
\begin{flalign*}
\| \Psi_t \circ \Psi_{t'}^{-1} - Id \|^2 &= \underset{\|u\|^2 + \|v\|^2 = 1}{\max}  \, \| (\Psi_t \circ \Psi_{t'}^{-1})(u,v) - (u,v) \|^2\\
&=   \left( \dfrac{\alpha(t)}{\alpha(t')}t - t' \right)^2 N^2 + \left( \dfrac{\alpha(t)}{\alpha(t')} - 1 \right)^2 \\
&\le \left[ \left( \dfrac{\alpha(t)}{\alpha(t')} - 1 \right) t + (t - t') \right]^2 N^2 + \left( \dfrac{\alpha(t)}{\alpha(t')} - 1 \right)^2 \\
&\le \left[ \left( \dfrac{\alpha(t)}{\alpha(t')} - 1 \right)\dfrac{1}{2} + (t - t') \right]^2 N^2 + \left( \dfrac{\alpha(t)}{\alpha(t')} - 1 \right)^2 \, .
\end{flalign*}
The last inequality follows since $0 \le t' < t \le 1/2$, which implies $(\alpha(t)/\alpha(t')) - 1 > 0$. It therefore suffices to prove that $(\alpha(t)/\alpha(t')) - 1 < C$ whenever $t-t' < C ((1/N) - (1/2))$, since this would give the sought-after estimate $\| \Psi_t \circ \Psi_{t'}^{-1} - Id \|^2 \le 2C^2 = \epsilon^2$. For this, we compute
\begin{flalign*}
\left( \dfrac{\alpha(t)}{\alpha(t')} - 1 \right) & \le \dfrac{1}{2} \left( \dfrac{\alpha(t)}{\alpha(t')} + 1 \right)\left( \dfrac{\alpha(t)}{\alpha(t')} - 1 \right) =  \dfrac{1}{2} \left( \left( \dfrac{\alpha(t)}{\alpha(t')}  \right)^2 - 1 \right) \\
&= \dfrac{N^2 [t(1-t) - t'(1-t')]}{2[1 - N^2 t(1-t)]} = \dfrac{N^2  (1-(t+t')) (t-t')}{2[1 - N^2 t(1-t)]} \\
& < \dfrac{N^2 (t-t') }{1 - (N^2/4)} < \dfrac{C ((1/N)-(1/2))}{ (1/N^2) -  (1/4)} = \dfrac{C}{ (1/N) +  (1/2)} < C \, .
\end{flalign*}
\end{proof}

%%%%%%%%%%%%%%%%%%%%%%%%%%%%%%%%%%%%%%%%%%%%%%%%%%%%%%%%%%%%%%%%%%%%%%%%%%%%%%%%
\section{Quantitative stability of the local tameness condition}\label{sec:stability}
%%%%%%%%%%%%%%%%%%%%%%%%%%%%%%%%%%%%%%%%%%%%%%%%%%%%%%%%%%%%%%%%%%%%%%%%%%%%%%%%

Let's consider a closed symplectic manifold $(M, \omega)$. Morally, our aim in this section is to get some quantitative control over the variation of the "degree of tameness" between two almost complex structure $J$ and $\psi^* J$ in terms of the $C^1$-size of the diffeomorphism $\psi$. The exact statement is given in Lemma \ref{lem:stability} below, whose formulation requires some preliminary setup.
\\

We fix a finite open atlas $\{(U_i, \phi_i : U_i \to \R^n)\}_{i=1}^N$ of $M$ by Darboux charts for $\omega$. Shrinking the open sets if necessary, we may and shall require that the pullbacks $g_i := \phi^* g_{std}$ of the standard Euclidean metric $g_{std}$ are all equivalent to $g|U_i$ for some given auxiliary Riemannian metric $g$ on $M$,\footnote{The symbol `$g$' was used to denote the canonical scalar product in Section \ref{sec:tion:canonical_metric}. Our recycling of the symbol is voluntary, as we shall pick during the proof of the preparation lemma an auxiliary Riemannian metric $g$ which coincides with the canonical scalar product along $Z$.} in the sense that there exists $K > 0$ such that $K^{-1} g_i(v,v) \le (g|U_i)(v,v) \le K g_i(v,v)$ for all $v \in TU_i$ and all $i$.

It is well known that there exists an open cover  $(V_i)_{i=1}^N$ of $M$ refining $(U_i)_{i=1}^N$ such that $\overline{V_i} \subset U_i$ for all $1 \le i \le N$. We fix such a refinement and we define
\[ \delta_c := \underset{1 \le i \le N}{\min} \, \dist_g( \overline{V_i}, M \setminus U_i) > 0 \, . \]
We shall assume that $\delta_c$ is smaller than both $1$ and the injectivity radius of the metric $g$ on $M$.

Given an almost complex structure $J$ on $U_i$, for each $1 \le i \le N$, we define the continuous function 
\[ \Upsilon_i[J] \, : \,  U_i \to \R \, : \, p \mapsto \Upsilon_i[J](p) := \underset{v \in (S_i)_p M}{\min} \, \omega_p(v, J_p v)   \]
where $(S_i)_p M := \{ v \in T_p U_i \, : \, (g_i)_p(v,v) = 1 \}$ is the sphere bundle defined by the metric $g_i$. We observe that $J$ is $\omega$-tame over the whole of $M$ if and only if all the functions $\Upsilon_i[J]$ are positive on their respective domains.

For each $1 \le i \le N$, given a bundle endomorphism $F$ of $TU_i$ (over the identity $id|U_i$) and a subset $S \subset U_i$, we set
\[  \| F \|_{C^0(S)} := \underset{p \in S}{\sup} \, \underset{v  \in S_p M}{\max} \; [(g_i)_p(F_p v, F_p v)]^{1/2} \, . \]
Typically, we shall consider $F = J$ or $F = dJ$, where $J$ is an almost complex structure and $dJ$ is computed with respect to the coordinates given by the chart $\phi_i$.

We shall consider the group $\Diff(M)$ of smooth diffeomorphisms of $M$ to be equipped with the $C^1$-Whitney topology (see~\cite[Chapter 2]{H}). In this context, we have:

\begin{lemma}\label{lem:stability} 
Given any constants $C > 0$ and $\eta > 0$, there exists a constant $\delta = \delta_{C, \eta} > 0$ such that the following holds:

\begin{enumerate}[(a)]
\item Let $B = B_{C, \eta} \subset \Diff(M)$ be the set of those diffeomorphisms satisfying
\[
\dist_g(p, \psi(p)) < \delta \text{~for all~} p \in M
\]
\[
\text{and~} \| T_p \psi - Id \|_{g_i} < \delta \text{~for all~} 1 \le i \le N \text{~and all~} p \in \overline{V_i}.
\]

\item Given any $\psi \in B$, set
\[
S:=S(\psi) = \{ q \in M ~:~ \dist_g(q,\, \overline{\supp(\psi)}) < \max_{x \in M} \dist_g (x, \psi(x)) \, \}.
\]

\item Given an almost complex structure $J$ on $M$ and a $\psi\in B$, let $\lambda:=\lambda(J,\psi) > 0$ be such that the estimates
\begin{flalign}
\notag  \| J \|_{C^0(S \cap U_i)}  <  \lambda C \; \text{~and~} \; \| d J \|_{C^0(S \cap U_i)} < \lambda C 
\end{flalign}
hold for all $1 \le i \le N$. 
\end{enumerate}
Then for all $1 \le i \le N$, we have
\[  \| \Upsilon_i[\psi^*J] - \Upsilon_i[J]\|_{C^0(V_i)} := \underset{p \in V_i}{\max} \, | \Upsilon_i[\psi^* J](p) - \Upsilon_i [J](p) | \, < \, \lambda \eta \, . \]
\end{lemma}
\begin{proof} We introduce a small parameter $0 < \delta' < \delta_c/4 < 1/4$, whose value we shall further constrain in the course of the proof, and we set $\delta := \delta'/(1+\delta')$, \emph{i.e.}, $\delta' = \delta/(1-\delta) < 1/3$. \\

\emph{Step 1 - Definition of $B$}. First, we restrict attention to the (open) set 
\[ B_0 := \{ \psi \in  \Diff(M) \, : \, \forall p \in M, \, \dist_g(p, \psi(p)) < \delta  \,  \} \, . \]
Clearly, $\psi \in B_0$ if and only if $\psi^{-1} \in B_0$. Given $\psi \in B_0$ and $p \in \overline{V_i}$, we observe that $\dist_g(\psi(p), M \setminus U_i) > 3 \delta_c /4$, so that $\psi(\overline{V_i}) \subset U_i$. Moreover, if $q \in M$ satisfies $\dist_g(p,q) < \delta_c/4$, then $\dist_g(\psi(p), \psi(q)) < 2 \delta' + \delta_c/4 \le 3 \delta_c/4$, so that $\psi(q) \in U_i$. 

Let's consider some $i$ and work in the chart $(U_i, \phi_i)$, thought of as a subset of $\R^n$ equipped with the standard metric $g_{std}$, so that the tangent spaces at different points of $U_i$ can be identified via the local connection $d$. Consider the (open) set
\[  B_i := \{ \psi \in B_0 \, : \, \forall p \in \overline{V_i}, \, \| T_p \psi - Id \|_{g_{std}} < \delta  \,  \} \, . \] 
Hence, given $p \in \overline{V_i}$ and $\psi \in B_i$, we have
\[  \| T_p \psi - Id \|_{g_{std}} < \delta' \; \text{~and~} \; \| T_{\psi(p)} (\psi^{-1}) - Id \|_{g_{std}} < \delta' \, . \]
We take $B := \bigcap_{i=1}^N B_i$ (which implicitly depends on the value of the parameter $\delta$).\\

\emph{Step 2 - Estimates on $\Upsilon$}. Consider some $\psi \in B$ and put  $\mu := \max_{p \in M} \dist_g (p, \psi(p))$, $S_0 := \overline{\supp(\psi)}$ and $S := \{ q \in M : \exists p \in S_0, \dist_g(p,q) < \mu \}$. (Of course, $\mu < \delta'$.) Consider an almost complex structure $J$ on $M$ and $\lambda > 0$ as in the statement (b) of the Lemma.

Let $p \in (M \setminus S_0) \cap V_i$. We can find a small open set $p \in U \subset (M \setminus S_0) \cap V_i$, so that $\psi|U = id|U$ and $\psi^*J|U = J|U$. Hence $\Upsilon_i[\psi^*J] - \Upsilon_i[J] = 0$ on such $U$.

Let $p \in S_0 \cap V_i$. We observe that the open $g$-ball $U$ of radius $\delta'$ centered at $p$ is contained in $S \cap U_i$, and that the image of this ball under $\psi$ lies inside $U_i$ and $\psi(p) \in U$. We may thus work in the Darboux chart $(U_i, \phi_i)$, which we shall think of as a subset of $\R^n$ equipped with the standard metric $g_{std}$. In particular, $\omega = \omega_{std}$ at every point of $U_i$. We note that $\dist_{g_{std}}(p, \psi(p)) < K \delta' =: L$ for all $p \in \overline{V_i}$. Since $\|dJ\| < \lambda C$ on $S \cap U_i$ and \emph{a fortiori} on $U$, integrating $dJ$ along the $g$-geodesic segment between $p$ and $\psi(p)$, we get $\| J_{\psi(p)} - J_{p} \| < \lambda CL$. This and $\|J_p\| < \lambda C$ yield $\| J_{\psi(p)}\| < \lambda C(L+1)$. For $v \in (S_i)_p M$, we compute
\begin{align}
\notag  \| \omega_{std} \|^{-1} \, \left| \omega_{\psi(p)}(v, (\psi^* J)_p v) - \omega_p(v, J_p v) \right| & \le  \| T_{\psi(p)}(\psi^{-1}) \cdot J_{\psi(p)} \cdot T_p \psi - J_p \| \\
\notag &\le \|  (T_{\psi(p)}(\psi^{-1}) - Id) \cdot J_{\psi(p)} \cdot (T_p \psi - Id) \|  \\
\notag  & \; \quad +  \| (T_{\psi(p)}(\psi^{-1}) - Id) \cdot J_{\psi(p)} \| \\
\notag & \; \quad +   \|   J_{\psi(p)} \cdot (T_p \psi - Id) \| + \|  J_{\psi(p)} - J_p \| \\
\notag & \le \lambda C(L+1) \delta^2 + 2 \lambda C(L+1) \delta + \lambda CL < \lambda C' \delta'
\end{align}
where $C' = C(4K+3)$ (recall that $\delta' \le 1$ by assumption). Let's now require $\delta' < \eta/C'  \| \omega_{std} \| $. Considering $v \in (S_i)_pM$ such that $\omega_p(v, J_p v) = \Upsilon_i[J](p)$ if $\Upsilon_i[\psi^*J](p) \ge \Upsilon_i[J](p)$ or such that $\omega_p(v, (\psi^*J)_p v) = \Upsilon_i[\psi^*J](p)$ if $\Upsilon_i[\psi^*J](p) < \Upsilon_i[J](p)$, we conclude
\[ \left| \Upsilon_i[\psi^* J](p) - \Upsilon_i[J](p) \right| < \lambda \eta \, .  \]
As this is true for all $p \in S_0 \cap V_i$ and for all $i$, this completes the proof.
\end{proof}

%%%%%%%%%%%%%%%%%%%%%%%%%%%%%%%%%%%%%%%%%%%%%%%%%%%%%%%%%%%%%%%%%%%%%%%%%%%%%%%%
\section{Proof of the Preparation lemma} \label{section:prep_lemma}
%%%%%%%%%%%%%%%%%%%%%%%%%%%%%%%%%%%%%%%%%%%%%%%%%%%%%%%%%%%%%%%%%%%%%%%%%%%%%%%%

We prove Theorem \ref{thm2} in this section.

%%%%%%%%%%%%%%%%%%%%%%%%%%%%%%%%%%%%%%%%%%%%%%%%%%%%%%%%%%%%%%%%%%%%%%%%%%%%%%%%
\subsection{Setup} \label{subsec:prep_setup}
%%%%%%%%%%%%%%%%%%%%%%%%%%%%%%%%%%%%%%%%%%%%%%%%%%%%%%%%%%%%%%%%%%%%%%%%%%%%%%%%

Let $(M, \omega)$ be a closed symplectic 4-manifold, $J$ be an $\omega$-tame almost complex structure, $\iota_Z : Z \subset M$ be an embedded closed $J$-holomorphic curve and $W$ an open neighborhood of $Z$ in $M$.

%%%%%%%%%%%%%%%%%%%%%%%%%%%%%%%%%%%%%%%%%%%%%%%%%%%%%%%%%%%%%%%%%%%%%%%%%%%%%%%%
\subsubsection{Linear level} \label{subsubsec:setup_lin_level}
%%%%%%%%%%%%%%%%%%%%%%%%%%%%%%%%%%%%%%%%%%%%%%%%%%%%%%%%%%%%%%%%%%%%%%%%%%%%%%%%
For each $p \in Z$, the data $(T_p M, T_p Z, \omega_p, J_p)$ is of the type studied in Section \ref{section:symplectic splitting}. Hence, as in Section \ref{sec:tion:canonical_metric}, we have the canonical scalar product $g_p$ at each $p \in Z$, thereby obtaining the \emph{canonical metric} $g$ on $T_Z M$ associated to $J$. Since our linear algebraic arguments depend smoothly on the given data, the metric $g$ is smooth. We can extend this bundle metric to a Riemannian metric on the whole of $M$; we fix such an auxiliary metric $g$. This metric is compatible with $\omega$ on $T_Z M$.

At each $p \in Z$, we have the number $N_J(p) \in [0,2)$ given by the norm of the skew part of $J_p$, which vanishes if and only if the tame pair $(\omega_p, J_p)$ is compatible  (see Section \ref{subsubsec:norm}). Allowing $p \in Z$ to vary, we thus obtain a continuous function $N_J : Z \to [0,2)$ whose maximum---which exists by compactness of $Z$---we denote $\|N_J\|$.

For $t \in [0,1/2]$, we consider the linear diffeotopy $\Psi_t : T_Z M \to T_Z M$ associated to $J$ defined in Section \ref{section:linear isotopy}, and we consider the corresponding path $J_t := \Psi_t^* J$ of $\omega$-tame almost complex structures on $T_Z M$. Recall that $J_{1/2}$ is in fact $\omega$-compatible.

In its essence, using Lemmata \ref{lemma:Whitney_diff} and \ref{lem:stability}, the following proof aims to extend $\Psi_t$ (more properly, a time discretization thereof) to an ambient diffeotopy $\psi_t : M \to M$ such that the almost complex structures $\psi_t^*J$ are $\omega$-tame over the whole of $M$.

\subsubsection{Darboux charts} Moving towards the setting of Section~\ref{sec:stability}, we want to construct a suitable finite open atlas  $\{(U_i, \phi_i : U_i \to \R^4)\}_{i=1}^N$ of $M$ by Darboux charts for $\omega$. This open cover will split into two parts: the first part $\{(U_j, \phi_j : U_j \to \R^4)\}_{j=1}^{N'}$ will cover some neighborhood of $Z$ in $M$ in some specific way, whereas the second part will simply be chosen to cover the rest of $M$ and avoid (a compact neighborhood of) $Z$. Only the first part will be of genuine interest to us, and we construct it as follows.

We first consider a finite open cover $(Y_j)_{j=1}^{N'}$ of $Z$ by Darboux charts for $\iota_z^*\omega$ over which the $\omega$-symplectic normal bundle $\nu$ of $Z$ (which is also the $g$-metric normal bundle of $Z$) trivializes both as a $\omega$-symplectic and as a $g$-metric vector bundle, \emph{i.e.}, for each $1 \le j \le N'$, we are given a vector bundle trivialization $\tau_j : \nu|Y_j \to Y_j \times \R^2$ such that $\tau_j^* \omega_{std} = \omega|\nu$ and $\tau_j^* g_{std} = g|\nu$ along the fibers of $\nu$. Roughly speaking, the following lemma proves that the data $\{(Y_j, \tau_j)\}_{j=1}^{N'}$ extend to Darboux charts $\{(U_j, \phi_j)\}_{j=1}^{N'}$ covering a neighborhood of $Z$ in $M$:

\begin{lemma}\label{lem:extension_Darboux}
Consider the data $\{(Y_j, \tau_j)\}_{j=1}^{N'}$ given above. There exist an open refinement $\{Z_j\}_{j=1}^{N'}$ of $\{Y_j\}_{j=1}^{N'}$ covering $Z$ and Darboux charts $\{(U_j, \phi_j : U_j \to \R^4)\}_{j=1}^{N'}$ for $\omega$ covering a neighborhood of $Z$ in $M$ such that $U_j \cap Z = Z_j$ and $T\phi_j | (\nu | Z_j) = \tau_j | Z_j$.
\end{lemma}

\textit{Proof}. We may think of the Darboux chart $Y_j$ as an open subset of $\R^2$ equipped with the standard symplectic form. We equip the set $M' := Y_j \times \R^2 \subset \R^4$ with the standard symplectic form and with the (compatible) standard metric. Of course, $Y_j \times \R^2$ can also be understood as the (symplectic and orthogonal) normal vector bundle $\nu'$ of $Y_j \cong Y_j \times \{0\}$ in $M'$. The trivialization $\tau_j$ can then be thought of as an isomorphism between the symplectic normal bundles $\nu$ of $Y_j$ in $M$ and $\nu'$ of $Y_j$ in $M'$ covering the symplectic diffeomorphism $id : Y_j \to Y_j$.

The result then follows from Weinstein's Symplectic neighborhood theorem (see \emph{e.g.} \cite[Theorem 3.4.10]{MS}). To apply this theorem, since the sets $Y_j$ are not compact submanifolds, we only need to shrink them a little to get compact submanifolds with boundaries $\overline{Z_j}$, whose interiors $Z_j$ we may require to still form an open cover of $Z$.

\hfill $\square$

\begin{remark} \label{rem:moser_trick}
In \cite{MS}, the Symplectic neighborhood theorem (Theorem 3.4.10) is deduced from the Moser isotopy lemma (Lemma 3.2.1). However, by itself, this lemma does not imply the part of the theorem which we use in the proof of Lemma \ref{lem:extension_Darboux} to get $T\phi_j | (\nu | Z_j) = \tau_j | Z_j$. Rather, it follows from the stronger equation (3.2.5) established in the proof of the Moser isotopy lemma.\eoe
\end{remark}

We now have the atlas  $\{(U_i, \phi_i : U_i \to \R^4)\}_{i=1}^N$ of $M$ by Darboux charts for $\omega$. We note that for those charts $U_j$ ($1 \le j \le N'$) that cover $Z$, the Darboux coordinates at points of $Z_j = Z \cap U_j$ determine a unitary basis as in Section~\ref{subsubsec:unitary-basis}, thereby allowing us to use the estimates proved in Section \ref{section:preliminary considerations}. We also fix some open refinement $\{(V_i)\}_{i=1}^N$ of $\{(U_i)\}_{i=1}^N$.

The atlas also determine the functionals $\Upsilon_i$ defined in Section \ref{sec:stability}. We thus have two criteria to test the tameness of a pair $(\omega, J')$ on $Z$ (assuming $Z$ is a $J'$-curve): either as the positivity of the functions $\Upsilon_i[J']$ or as the positivity of the function $2 - N_{J'}$. In view of Equation \ref{eq:tameness_N}, we see that for $p \in Z_j$, we have $\Upsilon_j[J'](p) \ge 1 - N_{J'}(p)/2$. 

\subsubsection{Whitney's extension theorem} 
We shall need the following "controlled" version of Whitney's extension theorem, whose proof is postponed to Section~\ref{section:Whitney_ext}: 

\begin{prop}[Whitney extension for diffeomorphisms]\label{prop:Whitney_diff_2}
Let $\Phi : T_ZM \to T_ZM$ be a bundle automorphism over the identity map $id : Z \to Z$ defining a holonomic $1$-jet data along $Z$ for diffeomorphisms of $M$, \emph{i.e.}, whose restriction $\left. \Phi \right|_{TZ} : TZ \to TZ$ is the identity. There exist constants $\epsilon_0 > 0$ and $\kappa \ge 1$ such that the following holds:

If $\| \Phi - Id \|_{C^0(Z \cap U_j)} < \epsilon_1$ for all $1 \le j \le N'$ and some $\epsilon_1 < \epsilon_0$, then there exists a diffeomorphism $\phi$ of $M$ which is diffeotopic to $id_M$ and whose differential along $Z$ equals $\Phi$. Moreover, $\phi$ can be chosen with $\max_{p \in M} \dist_g(p, \phi(p))$ as small as desired, and such that $\| d\phi - Id  \|_{C^0(\overline{V}_j)} < \kappa \epsilon_1$ for all $1 \le j \le N$. Furthermore, the diffeomorphism $\phi$ can be chosen such that its $2$-jets along $Z$ coincide with any given holonomic $2$-jet extension of $\Phi$ along $Z$.
\end{prop}

\begin{remark}\label{rem:isotopy_2}
As explained in Remark \ref{rem:isotopy}, the diffeotopy between $\phi$ and $id_M$ may be taken so as to restrict to the identity on $Z$ at all times.\eoe
\end{remark}

\subsubsection{Thin neighborhood $W$}\label{subsubsec:W}  We fix a relatively compact open neighborhood $W$ of $Z$ as in the statement of Theorem \ref{thm2}. We may assume that $W$ is contained within the union of the charts $V_j$ ($1 \le j \le N'$) that cover $Z$ and that $W$ avoids the charts $U_j$ ($j > N'$).

%%%%%%%%%%%%%%%%%%%%%%%%%%%%%%%%%%%%%%%%%%%%%%%%%%%%%%%%%%%%%%%%%%%%%%%%%%%%%%%%
\subsection{Choice of parameters} 
%%%%%%%%%%%%%%%%%%%%%%%%%%%%%%%%%%%%%%%%%%%%%%%%%%%%%%%%%%%%%%%%%%%%%%%%%%%%%%%%
We select appropriate values of the free parameters that appear in the Lemmata that we shall use later in the proof.

\subsubsection{Choice of $\eta$ in Lemma \ref{lem:stability} }\label{subsubsec:eta}  We consider the path $J_t$ of $\omega$-tame almost complex structures on $T_Z M$ from paragraph \ref{subsubsec:setup_lin_level}. We introduce a small parameter $\eta > 0$ so that $\eta < (1 - \|N_{J}\|/2)/2$, or equivalently so that $\|N_{J}\| < 2(1- 2\eta)$, which is possible since $J = J_0$ is $\omega$-tame. We recall from Section \ref{section:linear isotopy} that the function $t \mapsto \| N_{J_t} \| := \max_{p \in Z} N_{J_t}(p)$ is decreasing for $t \in [0,1/2]$, so that $\|N_{J_t}\| < 2(1- 2\eta)$ for all $t \in [0, 1/2]$.

\subsubsection{Choice of $C$ in Lemma \ref{lem:stability} }\label{subsubsec:C} We proceed here to fix an appropriately value for the constant $C$ in Lemma \ref{lem:stability}. First, we can certainly pick $C$ so large that
\[  \| J_0\|_{C^0(W \cap U_j)} < C \text{~and~} \| dJ_0\|_{C^0(W \cap U_j)} < C \]
for all $1 \le j \le N'$.

Secondly, let's imagine that the linear diffeotopy $\Psi_t$ of $T_Z M$ extends to an ambient diffeotopy $\widetilde{\psi}_t$ of $M$, and consider the corresponding almost complex structures $\widetilde{J}_t = \widetilde{\psi}_t^* J$. It is clear that the quantities $(\widetilde{J}_t)_{U_j}$ and $(d\widetilde{J}_t)_{U_j}$ are continuous functions on $U_j \times [0,1/2]$, and that their values along $Z$ are expressible in terms only of the $1$-jets of $J$ along $Z$ and of the $2$-jets of $\widetilde{\psi}_t$ along $Z$ (\emph{i.e.}, of some $2$-jet extension of $\Psi_t$); for instance, we know that $\widetilde{J}_t|Z = J_t|Z = \Psi_t^* J|Z$. We could then select $C$ to be larger than all of the norms $\| \widetilde{J}_t\|_{C^0(Z \cap U_j)}$ and $\| d\widetilde{J}_t\|_{C^0(Z \cap U_j)}$ for $1 \le j \le N'$ and $t \in [0,1/2]$.

Since we have not proved the existence of any such extensions $\widetilde{\psi}_t$ at this point, we shall instead select an arbitrary $2$-jet extension of $\Psi_t$ continuous over $Z \times [0,1/2]$, e.g., we could impose that the (well-defined) normal-normal part of the $2$-jet extension of the $1$-jet $\Psi_t$ (which is the only part of the $2$-jet extension which is not already determined by the $1$-jet $\Psi_t$) to vanish identically on $Z \times [0,1/2]$. Using the above expressions---that could be made explicit, but we only need their existence---to \emph{define} the quantities $J_t|Z$ (already given in fact by $\Psi_t^* J|Z$) and $dJ_t$ along $Z$ for all $t \in [0,1/2]$, it becomes clear that we can take $C$ so large that
\[  \| J_t\|_{C^0(Z \cap U_j)} < C \text{~and~} \| dJ_t\|_{C^0(Z \cap U_j)} < C\]
for all $1 \le j \le N'$ and all $t \in [0,1/2]$.

\subsubsection{Choice of $\epsilon$ in Lemma \ref{lem:newtechnical}} \label{subsubsec:epsilon} On the one hand, given the value $C$ that we just fixed and the value $\eta$ that we selected earlier, Lemma \ref{lem:stability} determines a parameter $\delta := \delta_{C, \eta} > 0$ and a corresponding open set $B = B_{C, \eta} \subset \Diff(M)$. On the other hand, under our current setup, Proposition \ref{prop:Whitney_diff_2} determines two parameters $\epsilon_0 > 0$ and $\kappa \ge 1$. We pick $0 < \epsilon < \min\{ \delta/\kappa,  \epsilon_0 \}$ (which is implicitly related to $\eta$).

\subsubsection{Choice of time partition}\label{subsec:Choice_of_partition} We partition the time-interval $[0,1/2]$ into $d$ intervals 
\[0 = t_0 < t_1 < \dots < t_d = 1/2\]
such that
\[ t_{i+1} - t_i <  \dfrac{\epsilon}{\sqrt{2}} \left( \frac{1}{\| N_J \|} - \frac{1}{2}  \right) \]
for all $0 \le i \le d-1$.

For $0 \le i \le d$, we set $\Psi_i := \Psi_{t_{i}}$ and $J_i = J_{t_i}$ for simplicity. It is also convenient to introduce the notation $\Phi_i = \Psi_{i+1} \circ \Psi_i^{-1}$ for $0 \le i \le d-1$; we note that each $\Phi_i$ comes with some induced $2$-jet extension. In view of Lemma \ref{lem:newtechnical}, for all $0 \le i \le d-1$, we have $\| \Phi_i  - Id \|_{C^0(Z)} < \epsilon$.

%%%%%%%%%%%%%%%%%%%%%%%%%%%%%%%%%%%%%%%%%%%%%%%%%%%%%%%%%%%%%%%%%%%%%%%%%%%%%%%%
\subsection{Proof of Theorem~\ref{thm2}} 
%%%%%%%%%%%%%%%%%%%%%%%%%%%%%%%%%%%%%%%%%%%%%%%%%%%%%%%%%%%%%%%%%%%%%%%%%%%%%%%%
The proof proceeds by induction over $i \in \{0, 1, \dots, d-1\}$.

\subsubsection{Induction hypothesis} We assume that for some $0 \le i < d$, we have constructed an ambient smooth diffeomorphism $\psi_i$ of $M$ whose $2$-jet along $Z$ coincides with the given $2$-jet extension of $\Psi_i$, and such that the globally defined almost complex structure $J_i := \psi_i^* J$ is everywhere $\omega$-tame. By assumption, this statement holds for $i=0$, since we can simply take $\psi_0 = id_M$.

\subsubsection{Induction step}  We focus our attention to within a thin open neighborhood $W_i$ of $Z$ inside the open set $W$ from Step \ref{subsubsec:W}, selected according to the following two constraints. First, along $Z$, we know that $\|(J_i)|_p\|_{g_j} < C$ and $\|(dJ_i)_p\|_{g_j} < C$ for all $p \in Z$ and $1 \le j \le N'$, where we used the notations $g_j := \tau_j^* g_{std}$. Since those norms depend continuously on $p \in M$, we can select $W_i$ so thin that we have $\|(J_i)|_p\|_{g_j} < C$ and $\|(dJ_i)_p\|_{g_j} < C$ for all $p \in W_i$ and $1 \le j \le N'$. Secondly, along $Z$, we know that $\| N_{J_i} \|  < 2 (1-2\eta)$, hence that $\Upsilon_j[J_i] > 2 \eta$ along $Z \cap U_j$ for all $1 \le j \le N'$. By continuity of $J_i$, we may select $W_i$ so thin that $\Upsilon_j[J_i] > \eta$ on $W_i \cap U_j$ for all $1 \le j \le N'$.

Given a subset $X \subset M$ and $\theta > 0$, write $X^{\theta} := \{ q \in M \, : \, \dist_g(q, X) < \theta\}$. Let's take $0 < \theta_i < \delta$ so small that $Z^{2\theta_i}  \subset W_i$. We set $Y_i := Z^{\theta_i}$, so that $Y_i^{\theta_i} \subseteq Z^{2 \theta_i}$.

We consider the bundle morphism $\Phi_i : T_Z M \to T_Z M$, which satisfies $\| \Phi_i - Id\|_{C^0(Z)} < \epsilon$, along with its corresponding $2$-jet extension. By choice of $\epsilon$ in Step \ref{subsubsec:epsilon} and in view of Proposition \ref{prop:Whitney_diff_2}, $\Phi_i$ and its $2$-jet extension extend to an ambient diffeomorphism $\phi_i \in \Diff(M)$ with compact support within $Y_i$, such that $\max_{p \in M} \dist_g(p, \phi_i(p)) < \theta_i$ and such that $\| d\phi_i - Id  \|_{C^0(\overline{V}_j)} < \delta$ for all $1 \le j \le N$. This means that $\phi_i$ belongs to the set $B$ obtained in Step \ref{subsubsec:epsilon}.

We set $\psi_{i+1} := \phi_i \circ \psi_i$ and $J_{i+1} := \phi_i^* J_i = \psi_{i+1}^* J$. It remains to prove that the almost complex structure $J_{i+1}$ is $\omega$-tame over the whole of $M$. Since $J_{i+1} = J_i$ on the complement of $W_i$, it suffices to prove that $J_{i+1}$ is $\omega$-tame on each $V_j$ with $1 \le j \le N'$.

Let's set $S_i := (\overline{\supp(\phi_i)})^{\theta_i}$. Hence $S_i \subseteq Y_i^{\theta_i} \subseteq Z^{2\theta_i} \subseteq W_i$ and thus
\[ \| J_i \|_{C^0(S_i \cap U_j)} < C \text{~and~} \| dJ_i \|_{C^0(S_i \cap U_j)} < C \, \]
for all $1 \le j \le N'$.

It follows that, with our choices of $\eta$ and $C$ from Steps \ref{subsubsec:eta} and \ref{subsubsec:C}, Lemma \ref{lem:stability} is applicable to the case of $J_i$ and $\phi_i$ (with $\lambda = 1$). Hence for all $1 \le j \le N'$, we have
\[  \left| \Upsilon_j[J_{i+1}] - \Upsilon_j[J_i]   \right|_{C^0(V_j)} < \eta \, . \]
Consequently, for all $1 \le j \le N'$ and all $p \in W_i \cap V_j$, we have
\[   \Upsilon_j[J_{i+1}](p) \ge \Upsilon_j[J_i](p) - \left| \Upsilon_j[J_{i+1}](p) - \Upsilon_j[J_i](p)   \right| > \eta - \eta = 0 \, , \]
proving that $J_{i+1}$ is $\omega$-tame at $p$. Since $W_i$ is covered by the $V_j$'s with $1 \le j \le N'$, this proves that $J_{i+1}$ is $\omega$-tame everywhere on $W_i$.

\subsubsection{Conclusion} In the end of the induction, we obtain a diffeomorphism $\psi = \psi_d$ (isotopic to the identity) and an $\omega$-tame almost complex structure $J' = J_d = \psi^* J$ whose restriction to $T_Z M$ is $\omega$-compatible. This concludes the proof of the Preparation lemma modulo the proof of Proposition~\ref{prop:Whitney_diff_2} that is given in the next section.

%%%%%%%%%%%%%%%%%%%%%%%%%%%%%%%%%%%%%%%%%%%%%%%%%%%%%%%%%%%%%%%%%%%%%%%%%%%%%%%%
\section{Whitney extension-type results}\label{section:Whitney_ext} 
%%%%%%%%%%%%%%%%%%%%%%%%%%%%%%%%%%%%%%%%%%%%%%%%%%%%%%%%%%%%%%%%%%%%%%%%%%%%%%%%

The Whitney extension-type result used in the proof of the preparation lemma --- namely Proposition \ref{prop:Whitney_diff_2} --- is a version of the more general Lemma~\ref{lemma:Whitney_diff} below reformulated within the specific setup of Section~\ref{section:prep_lemma}. (See Remark~\ref{rem:explanations_Whitney} for an explanation of the transition from Lemma~\ref{lemma:Whitney_diff} to Proposition~\ref{prop:Whitney_diff_2}.) Although these kinds of extension results may be well-known to experts, we have not been able to find anywhere in the literature statements giving sufficient control both over the $C^1$ norms and over the supports of the extended diffeomorphisms. For this reason, we provide a proof of Lemma~\ref{lemma:Whitney_diff} in this section.

We consider a closed $n$-manifold $M$ and a closed embedded $k$-submanifold $Z \subset M$ of positive codimension. Given a vector bundle $\nu : E \to M$, we shall denote by $\nu_Z : E_Z \to Z$ its pullback or restriction to $Z$. We say that a $r$-jet data along $Z$ (for some type of smooth maps from $M$ to some other manifold) is \emph{holonomic} if it satisfies suitable compatibility relations for it to stand a chance to be the $r$-jet of some genuine smooth map defined near $Z$, namely that the parts of the $r$-jet involving the tangent directions to $Z$ are obtained as derivatives (along $Z$) of lowest order parts of the $r$-jet (so that only the parts of the $r$-jet purely transverse to $Z$ can be chosen freely).

Our goal in this section is to establish some versions of Whitney extension theorem on manifolds, that is, roughly speaking, to find sufficient conditions for some $1$-jets data along $Z$ to be the differential along $Z$ of some function on $M$. The main result of this section is:

\begin{lemma}[Whitney extension for diffeomorphisms]\label{lemma:Whitney_diff}
Let $\Psi : T_ZM \to T_ZM$ be a bundle automorphism over the identity map $id : Z \to Z$ defining a holonomic $1$-jet data along $Z$ for diffeomorphisms of $M$, \emph{i.e.}, whose restriction $\left. \Psi \right|_{TZ} : TZ \to TZ$ is the identity. The following holds:

If $\Psi$ is sufficiently close to the identity bundle morphism, there exists a diffeomorphism $\psi$ of $M$ which is diffeotopic to $id_M$ and whose differential along $Z$ equals $\Psi$. Moreover, $\psi$ can be chosen to be ‘proportionally’ sufficiently $C^1$-close to the identity, arbitrarily $C^0$-close to the identity and supported in an arbitrary open set $U \supset Z$. Furthermore, the diffeomorphism $\psi$ can be chosen such that its $2$-jets along $Z$ coincide with any given holonomic $2$-jet extension of $\Psi$ along~$Z$.
\end{lemma}

\begin{remark}\label{rem:isotopy}
Recall that for any $\psi$ sufficiently $C^1$-close to the identity, $\psi$ can be interpreted as a (sufficiently $C^1$-small) smooth section of $T^*\Delta$ where $\Delta \subset M \times M$ denotes the diagonal. By a linear interpolation between this section and the zero section of $T^*\Delta$, we obtain a diffeotopy between $\psi$ and the identity. If $\psi$ restricts to the identity on $Z \subset M$, so does this whole diffeotopy.\eoe
\end{remark}

To prove the previous lemma, we shall show that it can be reformulated as a particular instance of the following :

\begin{lemma}[Whitney extension for sections]\label{lem:Whitney_sec}
Let $\nu : E \to M$ be a vector bundle of rank $q$ and let $F : T_Z M \to E_Z$  be a bundle morphism over the identity $id : Z \to Z$ defining a holonomic $1$-jet data along $Z$ for sections of $\nu$, \emph{i.e.}, which vanishes on $TZ \subset T_Z M$. The following holds:

There exists a smooth section $f : M \to E$ of $\nu$ which vanishes on $Z$ and whose differential along $Z$ is $F$. Moreover, $f$ can be chosen to be arbitrarily $C^0$-small and supported in an arbitrary open set $U \supset Z$. Besides, if $F$ is close to the zero morphism, then $f$ can be taken ‘proportionally’ $C^1$-small. Furthermore, the section $f$ can be chosen such that its $2$-jets along $Z$ coincide with any given holonomic $2$-jet extension of $F$ along $Z$.
\end{lemma}

\begin{proof}[Proof of Lemma \ref{lemma:Whitney_diff} assuming Lemma \ref{lem:Whitney_sec}.] 
Fix a Riemannian metric $g$ on $M$. Consider the diagonal $\Delta := \{ (x,x) \in M \times M \, | \, x \in M \}$ and $\Delta_Z := \Delta \cap (Z \times Z)$. It is well-known that its tangent bundle $T\Delta$ and its normal bundle $\nu_{\Delta} : E \to M$ in $M \times M$ are both isomorphic to $TM$ (in a canonically way for $T\Delta$ and in standard way for $\nu_{\Delta}$ using the metric $g$). Moreover, the Riemannian exponential map $exp : E \to M \times M$ determines a diffeomorphism between some sufficiently small tubular neighborhoods of the zero section of $\nu_{\Delta}$ in $E$ and of the diagonal $\Delta$ in $M \times M$. Through $exp$, sufficiently $C^1$-small smooth sections of $\nu_{\Delta}$ exactly correspond to the (graphs of the) diffeomorphisms of $M$ that are sufficiently $C^1$-close to the identity.

We now consider the bundle automorphism $\Psi$ from Lemma \ref{lemma:Whitney_diff}. The graph of $\Psi$ is the bundle monomorphism
\[ \gr \, \Psi : T_Z M \to T_Z M \oplus T_Z M : v \in T_p M \mapsto (v, \Psi(v)) \in T_p M \oplus T_pM \, . \]
By the previous bundle isomorphisms and in view of the canonical isomorphism $T_{\Delta_Z}(M \times M) \cong T_{\Delta_Z}M \oplus T_{\Delta_Z}M \cong T_Z M \oplus T_Z M$, we can interpret $ \gr \, \Psi$ as the bundle map
\[  \gr \, \Psi : T_{\Delta_Z} \Delta \to T_{\Delta_Z}(M \times M) : (v,v) \in T_{(p,p)}\Delta \mapsto (v, \Psi(v)) \in T_{(p,p)}(M \times M) \, .  \]
Postcomposing this map with the projection $T_{\Delta_Z}(M \times M) \to (\nu_{\Delta})_{\Delta_Z}$, we see that $\Psi$ determines a bundle morphism $F : T_{\Delta_Z}\Delta \to (\nu_{\Delta})_{\Delta_Z}$ which vanishes on $T(\Delta_Z) \subset  T_{\Delta_Z}\Delta$ and which is sufficiently $C^1$-small. Lemma \ref{lem:Whitney_sec} then implies the existence of a ‘proportionally’ $C^1$-small section $f : \Delta \to E$ which vanishes on $\Delta_Z$ and whose differential along $\Delta_Z$ is $F$. For sufficiently $C^1$-small $\Psi$, this section $f$ is sufficiently $C^1$-small to correspond -- under $\exp$ -- to the sought-after diffeomorphism $\psi$ of $M$. Clearly, if $f$ is $C^0$-small (respectively, supported in $U$), then $\psi$ is $C^0$-close to the identity (respectively, supported in $U$).

Finally, it is clear that there is a correspondence between holonomic $2$-jet extensions of $\Psi$ and holonomic $2$-jet extensions of $F$, and similarly a correspondence between $2$-jets of $\psi$ and $2$-jets of $f$ along $Z$.
\end{proof}

We now turn to the proof of Lemma \ref{lem:Whitney_sec}.  The general strategy is to use Whitney's original extension theorem \cite{W} to prove a local version of the Lemma and then to use a partition of unity argument to deduce the global version. For convenience, we recall here the statement of Whitney's theorem (essentially as given in \cite[Section 1.5.6]{N}, referring to \cite{W} or to \cite[Chapter 1.4]{M} for its proof):

\begin{theorem}[Whitney's extension theorem for $C^{\infty}$-functions]\label{thm:Whitney}
Let $\Omega$ be an open set in $\R^n$ and $X$ be a closed subset of  $\Omega$. Suppose that for each $n$-tuple $\alpha = (\alpha_1, \dots, \alpha_n)$ of non-negative integers, there is given a continuous function $f_{\alpha}$ on $X$. Then the following statements are equivalent:
\begin{enumerate}
\item There exists $f \in C^{\infty}(\Omega)$ for which $\left. D^{\alpha}f \right| X = f_{\alpha}$ for all $\alpha$;

\item For any $\alpha$, any integer $m \ge 0$ and any compact set $K \subset X$, it holds that
\begin{flalign}\label{eq:Whitney}
f_{\alpha}(x) = \sum_{|\beta| \le m} \dfrac{1}{\beta !} f_{\alpha + \beta}(y) (x-y)^{\beta} \, + \, o(|x-y|^m)  
\end{flalign}
uniformly as $|x-y| \to 0$ with $x,y \in K$.
\end{enumerate}
\end{theorem}\vspace{12pt}

\begin{proof}[Proof of Lemma \ref{lem:Whitney_sec}.]~\\

\emph{Step 1 - Local existence.} Let $p \in Z$ and consider a small open chart $p \in V \subset M$ centered at $p$ such that  $\nu$ trivializes over $V$. In this way, we can assume that $Z$ is a (relatively) closed embedded submanifold in an open ball $M = V$ of $\R^n$, that $E = V \times \R^q$ and that $F : Z \times \R^n \to Z \times \R^q$, $F(z,v) = (z, F^{(1)}(z,v))$ where $F^{(1)}$ is a $z$-wise $v$-linear map thought of as a formal first-order differential. Let's only consider the more difficult case when a $2$-jet extension of $F$ is also given, \emph{i.e.}, a bundle map $\widetilde{F} : Z \times \R^n \times (\R^n \times \R^n) \to Z \times \R^q$, 
$$\widetilde{F}(z, v, w^{(1)}, w^{(2)}) = (z, F^{(1)}(z,v) + F^{(2)}(z, w^{(1)}, w^{(2)}))\, , $$
where $F^{(2)}$ is a $z$-wise $w$-bilinear map thought of as a formal Hessian operator. Working with each component of $\R^q$ separately, we may assume $q = 1$.

In fact, we can further assume that $Z$ is given in the local coordinates $x_1, \dots, x_n$ by the equations $x_{k+1} = \dots = x_n = 0$. Since each of the $\R^n$ factors intervening in the domains of $F$ and $\widetilde{F}$ are to be interpreted as the tangent space of $M$ at a point of $Z$, we parametrize them with the coordinates $v_1, \dots, v_n$, $w_1^{(1)}, \dots, w_n^{(1)}$ and $w_1^{(2)}, \dots, w_n^{(2)}$ respectively, all thought of as being `the same as' the coordinates $x_1, \dots, x_n$. For this reason, for $p \in Z$, $T_p Z$ is given in each of the three $\R^n$ by setting the last $n-k$ coordinates equal to $0$, e.g., $T_pZ$ is given in the first $\R^n$ factor by the equations $v_{k+1} = \dots = v_n = 0$.

By the assumptions on $F$, we have that $F^{(1)}(x,v) = \sum_{j=k+1}^n f_j(x)v_j$ for some smooth functions $f_j : Z \to \R$. It is convenient to set $f_j = 0$ for $1 \le j \le k$. Similarly, we have $ F^{(2)}(x, w^{(1)}, w^{(2)}) = \sum_{r,s = 1}^n f_{rs}(x) w_r^{(1)} w_s^{(2)}$ for some smooth functions $f_{rs} : Z \to \R$. Moreover, since $\widetilde{F}$ is a holonomic $2$-jet extension of $F$, we have:
\begin{enumerate}[(i)]
    \item $f_{rs} = f_{sr}$ for all $1 \le r, s \le n$;

    \item $f_{rs}(x) = (\partial_{x_r} f_s)(x)$ for all $1 \le r \le k$ and $1 \le s \le n$.
\end{enumerate}
\noindent It follows in particular that $f_{rs} = 0$ for $r, s \le k$.

Now, for each multi-index $\alpha = (\alpha_1, \dots, \alpha_n)$, we select a function $f_{\alpha} : Z \to \R$ as follows (we use the notations $|\alpha| = \sum_{i=1}^n \alpha_i$, $\alpha_{\le k} = (\alpha_1, \dots, \alpha_k, 0, \dots, 0)$ and $\alpha_{> k} = (0, \dots, 0, \alpha_{k+1}, \dots, \alpha_n)$):
\begin{enumerate}
\item For $\alpha = (0)$, $f_{(0)} = 0$.
\item When $|\alpha| = 1$: for $\alpha_{(j)} = (\alpha_{(j)i})_{1 \le i \le n}$ where $\alpha_{(j)i} := \delta_{ji}$ (the Kronecker delta) and $1 \le j \le n$, we set $f_{\alpha_{(j)}}(x):= f_j(x)$.
\item When $|\alpha| = 2$: for $\alpha_{(rs)} := \alpha_{(r)} + \alpha_{(s)}$ where $1 \le r,s \le n$, we set $f_{\alpha_{(rs)}}(x) :=  f_{rs}(x) = f_{sr}(x)$.
\item When $|\alpha| \ge 3$: if $|\alpha|_{\le k} = 0$, we may pick any function for $f_{\alpha}$, say $f_{\alpha} = 0$; otherwise, we set $f_{\alpha} = \partial_{\alpha_{\le k}} f_{\alpha_{> k}}$.
\end{enumerate}
\noindent We observe that the identity $f_{\alpha} = \partial_{\alpha_{\le k}} f_{\alpha_{> k}}$ holds in fact for every $\alpha$, and that $f_{\alpha} = 0$ identically on $Z$ whenever $\alpha_{> k} = (0)$. 

We claim these choices satisfy Condition \ref{eq:Whitney}. Indeed, since we need to consider points $x,y \in K \subset Z$, the product $(x-y)^{\beta}$ vanishes whenever $\beta_{>k} \neq (0)$. Hence, for any $\alpha$, any integer $m \ge 0$ and any compact set $K \subset Z$, it holds that
\begin{flalign*}\label{eq:Whitney}
f_{\alpha}(x) &= (\partial_{\alpha_{\le k}} f_{\alpha_{>k}})(x) \\
&= \sum_{|\beta| \le m, \, \beta_{>k} = (0)} \dfrac{1}{\beta !} (\partial_{\alpha_{\le k} + \beta} f_{\alpha_{>k}})(y) (x-y)^{\beta} \, + \, o(|x-y|^m)  \\
&= \sum_{|\beta| \le m, \, \beta_{>k} = (0)} \dfrac{1}{\beta !} f_{\alpha + \beta}(y) (x-y)^{\beta} \, + \, o(|x-y|^m) \\
&= \sum_{|\beta| \le m} \dfrac{1}{\beta !} f_{\alpha + \beta}(y) (x-y)^{\beta} \, + \, o(|x-y|^m)
\end{flalign*}
uniformly as $|x-y| \to 0$ with $x,y \in K$, where the second equality follows from Taylor's theorem and the last equality follows since the product $(x-y)^{\beta}$ vanishes whenever $\beta_{>k} \neq (0)$ and $x,y \in K \subset Z$.

Hence Condition \ref{eq:Whitney} is fully established. Theorem \ref{thm:Whitney} thus implies the existence of a smooth function $f : V \to \R$ that vanishes on $Z$ and whose differential along $Z$ equals $F$. \\

\emph{Step 2 - Global existence}. For each $p \in Z$, consider a small open set $V_p$ centered at $p$ as in \emph{Step 1} and denote $f_p : V_p \to E$ the corresponding local solution. Since $Z$ is compact, there is a thin compact neighborhood $Z' \supset Z$ covered by sets $V_1 := V_{p_1}, \dots, V_m := V_{p_m}$, and the collection $\mathcal{V} := (V_0 := M \setminus Z', V_1, \dots, V_m)$ is a finite open cover of $M$. Let $(\chi_0, \chi_1, \dots, \chi_m)$ be a smooth partition of unity subordinated to $\mathcal{V}$; we note that $\sum_{j=1}^m \chi_j(x) = 1$ on $Z'$.

Define the smooth section $f : M \to E$ by $f(x) := \sum_{j=1}^m \chi_j(x) f_j(x)$. It clearly vanishes on $Z$. By \emph{Step 1}, at each point $p \in Z$, the $2$-jets of the functions $f_j$ defined at this point are all equal. Hence, for $p \in Z$, we compute in any chart containing $p$:
\[ df_p = \sum_{j=1}^m [d(\chi_j)_p \, f_j(p) + \chi_j(p) \, d(f_j)_p] =  \sum_{j=1}^m [0 + \chi_j(p) \, F_p] = F_p \, ,  \]
and similarly
\begin{flalign*}
d^2f_p &= \sum_{j=1}^m [d^2(\chi_j)_p \, f_j(p) + 2 d(\chi_j)_p \, d(f_j)_p + \chi_j(p) \, d^2(f_j)_p] \\
&=  \sum_{j=1}^m [d^2(\chi_j)_p \, 0 + 2 d(\chi_j)_p \, F_p + \chi_j(p) \, F^{(2)}_p] = F^{(2)}_p \, .
\end{flalign*}
This proves the existence of a section $f$ extending the given holonomic $2$-jet data along $Z$.\\

\emph{Step 3 - $C^0$ and $C^1$ control}. Replacing $f$ by its multiplication with a bump function that equals $1$ in a neighborhood of $Z$ and that is supported in a thin tubular neighborhood of $Z$, we may ensure that $f$ is supported in any neighborhood $U \supset Z$. Furthermore, since $f$ vanishes on $Z$ and $Z$ is compact, by taking $U$ thin enough, $f$ can be made arbitrarily $C^0$-small.

To prove the last claim about the $C^1$-smallness of the extension, fix some metrics on $M$ and $E$ and consider the corresponding metric connection $\nabla$. Define
\[ K :=  \underset{p \in Z}{\max} \, \underset{v \in T_p M \, , \, \|v \| = 1 }{\max} \, \| F(p, v) \| \, . \]
For $p \in M$, let $r(p) := \dist(p, Z)$ denote the geodesic distance from $p$ to $Z$.

Fix an extension $f$ of $F$ as given above and consider the tubular neighborhood $U$ of $Z$ of geodesic radius $R$. Since $M$ is compact and $f$ is smooth, there exists $C > 0$ such that $\| (\nabla_v f)(p) \| \le K + C r(p)$ for all $p \in M$ and for all $v \in T_p M$ of norm $1$. Hence, by taking $R$ small enough and since $f$ vanishes on $Z$, we obtain that $\|f(p)\| \le 2K r(p)$ for all $p \in U$. For later convenience, we shall also assume that $R < K/C$.

Let $\rho : [0, + \infty) \to [0,1]$ be a smooth non-increasing function supported in $[0, R]$ such that $\rho = 1$ near $0$ and which is approximately affine on $[0,R]$, so that $|\rho'(x)| < 2/R$. Then the function $\tilde{f}(p) := \rho(r(p)) f(p)$ still extends $F$ as in the Lemma and is supported in $U$. Moreover, given $p \in U$ and $v \in T_p M$ of norm $1$, we estimate 
\begin{flalign*}
 \| (\nabla \tilde{f}_p)(v) \| &\le | (d\rho)_p(v)| \,\| f(p) \| + |\rho(p)| \, \| (\nabla_v f)(p) \| \\
\notag & \le (2/R) \, 2K r(p) +  (K + C r(p))  \le  5 K + CR \le 6K \, .
\end{flalign*}
This proves that if $F$ is close to the zero morphism, \emph{i.e} if $K$ is small, then there exists an extension $\tilde{f}$ of $F$ which is ‘proportionally’ $C^1$-small over the whole of $M$.
\end{proof}

\begin{remark}\label{rem:explanations_Whitney}
Proposition~\ref{prop:Whitney_diff_2} is a quantitative formulation of Lemma~\ref{lemma:Whitney_diff} in terms of the $C^0$ and $C^1$ norms on diffeomorphisms and jets that the specific setup of Section~\ref{section:prep_lemma} allows to define. Only the meaning of the constants $\epsilon_0 > 0$ and $\kappa \ge 1$ deserve further explanation.

Lemma~\ref{lemma:Whitney_diff} applies to any $1$-jet $\Phi$ along $Z$ which is sufficiently close to the identity bundle morphism. Since $M$ is compact, this applies whenever the (finitely many) values $\| \Phi - Id\|_{C^0(Z \cap U_j)}$ are all smaller than some appropriate constant $\epsilon_0 > 0$. (We recall that the need for considering $1$-jets that are sufficiently close to the identity morphism stems from the way we reduce Lemma~\ref{lemma:Whitney_diff} to Lemma~\ref{lem:Whitney_sec}.)

Next, take $\epsilon_1$ ($< \epsilon_0$) to be the maximum of the norms $\| \Phi - Id\|_{C^0(Z \cap U_j)}$. Lemma~\ref{lemma:Whitney_diff} implies that $\Phi$ extends to a diffeomorphism $\phi$ over $M$ which is "proportionally $C^1$-close" to the identity diffeomorphism. In terms of norms, this means that there exists a proportionality constant $\kappa \ge 1$ (independent of $\Phi$ and $\phi$) such that the (finitely many) values $\| d\phi - Id \|_{C^0(\overline{V}_j)}$ are all smaller than $\kappa \epsilon_1$. The existence of this constant $\kappa$ ultimately follows from our proof of Lemma~\ref{lem:Whitney_sec}: Concretely, bearing in mind the transition from diffeomorphisms to sections between Lemmata~\ref{lemma:Whitney_diff} and ~\ref{lem:Whitney_sec} and the changes between equivalent norms when going from the setup of Proposition~\ref{prop:Whitney_diff_2} to that of Step $3$ in our proof of Lemma~\ref{lem:Whitney_sec}, the parameters $\epsilon_1$ and $\kappa$ in Proposition~\ref{prop:Whitney_diff_2} correspond respectively to the parameter $K$ and to the proportionality constant $6$ derived in Step $3$ in our proof of Lemma~\ref{lem:Whitney_sec}.
\eoe
\end{remark}

%%%%%%%%%%%%%%%%%%%%%%%%%%%%%%%%%%%%%%%%%%%%%%%%%%%%%%%%%%%%%%%%%%%%%%%%%%%%%%%%
\section*{Declarations}
%%%%%%%%%%%%%%%%%%%%%%%%%%%%%%%%%%%%%%%%%%%%%%%%%%%%%%%%%%%%%%%%%%%%%%%%%%%%%%%%

%\item Funding
\subsection*{Conflict of interest} All authors declare that they have no conflicts of interest.
%\item Ethics approval 
%\item Consent to participate
%\item Consent for publication
\subsection*{Availability of data} No datasets were generated or analyzed for the current work.
%\item Code availability 
\subsection*{Authors' contributions} The authors contributed equally to this work.

%%%%%%%%%%%%%%%%%%%%%%%%%%%%%%%%%%%%%%%%%%%%%%%%%%%%%%%%%%%%%%%%%%%%%%%%%%%%%%%%
% BIBLIOGRAPHY
%%%%%%%%%%%%%%%%%%%%%%%%%%%%%%%%%%%%%%%%%%%%%%%%%%%%%%%%%%%%%%%%%%%%%%%%%%%%%%%%

\end{document}